\documentclass[a4paper]{amsart}
\usepackage{amsthm}
\usepackage{amssymb}
\usepackage{dsfont}
\usepackage{enumerate}
\usepackage{hyperref}
\usepackage{tikz-cd} 

\numberwithin{equation}{section}

\theoremstyle{plain}

\newtheorem{lemma}{Lemma}[section]
\newtheorem{theorem}[lemma]{Theorem}
\newtheorem{proposition}[lemma]{Proposition}
\newtheorem{corollary}[lemma]{Corollary}

\theoremstyle{definition}
\newtheorem{construction}[lemma]{Construction}

\newtheorem{example}[lemma]{Example}

\newtheorem*{ack}{Acknowledgements}

\theoremstyle{remark} 
\newtheorem{remark}[lemma]{Remark} 
\newtheorem*{claim}{Claim} 

\newcommand{\Coker}{\operatorname{Coker}}

\renewcommand{\dim}{\operatorname{dim}}
\newcommand{\End}{\operatorname{End}}
\newcommand{\Ext}{\operatorname{Ext}}
\newcommand{\tExt}{\operatorname{\rlap{$\smash{\widehat{\mathrm{Ext}}}$}\phantom{\mathrm{Ext}}}}

\newcommand{\Hom}{\operatorname{Hom}}

\newcommand{\PHom}{\operatorname{PHom}}
\newcommand{\id}{\operatorname{id}}

\newcommand{\Inj}{\operatorname{\mathsf{Inj}}}
\newcommand{\KacInj}[1]{\mathsf K_{\mathrm{ac}}(\Inj #1)}

\newcommand{\KInj}[1]{\mathsf K(\Inj #1)}

\renewcommand{\mod}{\operatorname{\mathsf{mod}}}
\newcommand{\Mod}{\operatorname{\mathsf{Mod}}}

\newcommand{\Proj}{\operatorname{Proj}}
\newcommand{\rank}{\operatorname{rank}}

\newcommand{\sEnd}{\underline{\End}}

\newcommand{\sHom}{\underline{\Hom}}
\newcommand{\Spec}{\operatorname{Spec}}

\newcommand{\StMod}{\operatorname{\mathsf{StMod}}}
\newcommand{\stmod}{\operatorname{\mathsf{stmod}}}

\newcommand{\supp}{\operatorname{supp}}

\newcommand{\Thick}{\operatorname{\mathsf{Thick}}}

\newcommand{\tors}{\gamma}

\newcommand{\comp}{\mathop{\circ}}
\newcommand{\col}{\colon}
\newcommand{\ges}{{\scriptscriptstyle\geqslant}}

\newcommand{\kos}[2]{{#1}/\!\!/{#2}}

\newcommand{\op}{\mathrm{op}}

\newcommand{\da}{{\downarrow}}
\newcommand{\lra}{\longrightarrow}

\newcommand{\xra}{\xrightarrow}

\newcommand{\mcF}{\mathcal{F}}

\newcommand{\mcV}{\mathcal{V}}

\newcommand{\sfc}{\mathsf c}

\newcommand{\sfi}{\mathsf i}
\newcommand{\sfp}{\mathsf p}

\newcommand{\sft}{\mathsf t}

\newcommand{\sfC}{\mathsf C}
\newcommand{\sfD}{\mathsf D}

\newcommand{\sfT}{\mathsf T} 
\newcommand{\sfU}{\mathsf U}

\newcommand{\bbZ}{\mathbb Z} 

\newcommand{\bsa}{\boldsymbol{a}} 
\newcommand{\bsb}{\boldsymbol{b}}

\newcommand{\fa}{\mathfrak{a}} 
\newcommand{\fb}{\mathfrak{b}}
 
\newcommand{\fm}{\mathfrak{m}} 
\newcommand{\fp}{\mathfrak{p}}
\newcommand{\fq}{\mathfrak{q}}

\newcommand{\gam}{\varGamma}

\def\Si{\Sigma}

\def\one{\mathds 1}

\newcommand{\Tr}{\operatorname{Tr}\nolimits}

\title[Local duality]{Local duality for representations of \\ finite group schemes}

\author[Benson, Iyengar, Krause, and Pevtsova]{Dave Benson, Srikanth
  B. Iyengar, Henning Krause \\ and Julia Pevtsova}

\address{Dave Benson \\ 
Institute of Mathematics\\ 
University of Aberdeen\\ 
King's College\\ 
Aberdeen AB24 3UE\\ 
Scotland U.K.}

\address{Srikanth B. Iyengar\\ 
Department of Mathematics\\
University of Utah\\ 
Salt Lake City, UT 84112\\ 
U.S.A.}

\address{Henning Krause\\ 
Fakult\"at f\"ur Mathematik\\ 
Universit\"at Bielefeld\\ 
33501 Bielefeld\\ 
Germany.}

\address{Julia Pevtsova\\ 
Department of Mathematics\\ 
University of Washington\\ 
Seattle, WA 98195\\ 
U.S.A.}

\begin{document}

\begin{abstract} 
  A duality theorem for the stable module category of representations
  of a finite group scheme is proved. One of its consequences is an
  analogue of Serre duality, and the existence of Auslander-Reiten
  triangles for the $\fp$-local and $\fp$-torsion subcategories of the
  stable category, for each homogeneous prime ideal $\fp$ in the
  cohomology ring of the group scheme.
\end{abstract}

\keywords{Serre duality, local duality, finite group scheme, stable
module category, Auslander-Reiten triangle} \subjclass[2010]{16G10
(primary); 20C20, 20G10 20J06, 18E30}

\date{12 July 2017}

\thanks{SBI was partly supported by NSF grant DMS-1503044 and JP was
partly supported by NSF grants DMS-0953011 and DMS-1501146.}

\maketitle

\section{Introduction} 
This work concerns the modular representation theory of finite groups
and group schemes. A starting point for it is a duality theorem for
finite groups due to Tate, that appears already in Cartan and
Eilenberg~\cite{Cartan/Eilenberg:1956a}. For our purposes it is useful
to recast this theorem in terms of stable module categories. Recall
that the stable module category of a finite group scheme $G$ over a
field $k$ is the category obtained from the (abelian) category of
finite dimensional $G$-modules by annihilating morphisms that factor
through a projective module; we denote it $\stmod G$, and write
$\sHom_{G}(-,-)$ for the morphisms in this category. The category
$\stmod G$ is triangulated with suspension $\Omega^{-1}$, and Tate
duality translates to the statement that for all finite dimensional
$G$-modules $M$ and $N$ there are natural isomorphisms
\[ \Hom_{k}(\sHom_{G}(M,N),k) \cong
\sHom_{G}(N,\Omega\delta_{G}\otimes_{k}M)\,.
\] Here $\delta_{G}$ is the modular character of $G$, described in
Jantzen \cite[\S I.8.8]{Jantzen:2003a}; it is isomorphic to the
trivial representation $k$ when $G$ is a finite group.  This statement
can be deduced from a formula of Auslander and
Reiten~\cite{Auslander:1978a} that applies to general associative
algebras; see Theorem~\ref{thm:tate}.

In the language introduced by Bondal and
Kapranov~\cite{Bondal/Kapranov:1990a} the isomorphism above says that
$\stmod G$ has Serre duality with Serre functor $M\mapsto
\Omega\delta_{G}\otimes_{k}M$. One of the main results of our work is
that such a duality also holds \emph{locally}.

The precise statement involves a natural action of the cohomology ring
$H^{*}(G,k)$ of $G$ on the graded abelian group
\[ \sHom_{G}^{*}(M,N)=\bigoplus_{n\in\bbZ}\sHom_{G}(M,\Omega^{-n}N)\,.
\] The ring $H^{*}(G,k)$ is graded commutative, and also finitely
generated as a $k$-algebra, by a result of Friedlander and
Suslin~\cite{Friedlander/Suslin:1997a}. Fix a homogeneous prime ideal
$\fp$ not containing $H^{\ges 1}(G,k)$ and consider the triangulated
category $\tors_{\fp}(\stmod G)$ that is obtained from $\stmod G$ by
localising the graded morphisms at $\fp$ and then taking the full
subcategory of objects such that the graded endomorphisms are
$\fp$-torsion; see Section~\ref{sec:serre} for details. Our interest
in the subcategories $\tors_{\fp}(\stmod G)$ stems from the fact that
they are the building blocks of $\stmod G$ and play a key role in the
classification of its tensor ideal thick subcategories; see
\cite{Benson/Iyengar/Krause:2015a}. These subcategories may thus be
viewed as analogues of the $K(n)$-local spectra in stable homotopy
theory that give the chromatic filtration of a spectrum;
see~\cite{Ravenel:1992a}.

Our local Serre duality statement reads:

\begin{theorem}
  \label{thm:main1} Let $\sfC:=\tors_{\fp}(\stmod G)$ and $d$ the
Krull dimension of $H^*(G,k)/\fp$. The assignment $M\mapsto
\Omega^{d}\delta_G\otimes_k M$ induces a Serre functor for
$\sfC$. Thus for all $M,N$ in $\sfC$ there are natural isomorphisms
\[
\Hom_{H^*(G,k)}(\Hom^{*}_{\sfC}(M,N),I(\fp))\cong\Hom_{\sfC}(N,\Omega^{d}\delta_G\otimes_k
M)
\] where $I(\fp)$ is the injective hull of the graded
$H^*(G,k)$-module $H^*(G,k)/\fp$.
\end{theorem}

One of the corollaries is that $\tors_{\fp}(\stmod G)$ has
Auslander-Reiten triangles. One can thus bring to bear the techniques
of AR theory to the study of $G$-modules.  These results are contained
in Theorem~\ref{thm:stmodAR}.

We deduce Theorem~\ref{thm:main1} from a more general result
concerning $\StMod G$, the stable category of all (including infinite
dimensional) $G$-modules. Consider its subcategory $\gam_{\fp}(\StMod
G)$ consisting of the $\fp$-local $\fp$-torsion modules; in other
words, the $G$-modules whose support is contained in $\{\fp\}$. This
is a compactly generated triangulated category and the full
subcategory of compact objects is equivalent, up to direct summands,
to $\tors_{\fp}(\stmod G)$; this is explained in
Remark~\ref{rem:compacts-localisation}. There is an idempotent functor
$\gam_\fp\colon \StMod G\to \StMod G$ with image the $\fp$-local
$\fp$-torsion modules; see
Section~\ref{sec:cohomology-and-localisation} for details. The central
result of this work is that $\gam_\fp(\delta_G)$ is a dualising object
for $\gam_{\fp}(\StMod G)$, in the following sense.

\begin{theorem}
\label{thm:main2} For any $G$-module $M$ and $i\in\bbZ$ there is a
natural isomorphism
  \[ \tExt_{G}^{i}(M,\gam_\fp(\delta_G)) \cong
\Hom_{H^{*}(G,k)}(H^{*-d-i}(G,M),I(\fp)).
  \]
\end{theorem}

This result, proved in Section~\ref{sec:Duality}, may be compared with
Serre duality for coherent sheaves $\mcF$ on a non-singular projective
variety $X$ of dimension $n$:
\[ \Ext^i_X(\mcF,\omega_X) \cong \Hom_k(H^{n-i}(X,\mcF),k),
\] see, for example, Hartshorne \cite{Hartshorne:1977a}.

When $G$ is a finite group $\gam_{\fp}(k)$ is the Rickard idempotent
module $\kappa_{V}$, introduced by Benson, Carlson, and
Rickard~\cite{Benson/Carlson/Rickard:1996a}, that is associated to the
irreducible subvariety $V$ of $\Proj H^{*}(G,k)$ defined by $\fp$. In
this context, Theorem~\ref{thm:main2} was proved by Benson and
Greenlees~\cite{Benson/Greenlees:2008a}; see the paragraph following
Theorem~\ref{thm:gorenstein} below for a detailed comparison with
their work, and that of Benson~\cite{Benson:2008a}.

Concerning $\gam_{\fp}(k)$, the following consequences of
Theorem~\ref{thm:main2} have been anticipated in \cite{Benson:2001a}
when $G$ is a finite group.

\begin{theorem} 
  Suppose that $\delta_G\cong k$.  Then the $H^*(G,k)$-module
  $H^{*}(G,\gam_{\fp}(k))$ is injective and up to a twist isomorphic
  to $I(\fp)$. Also, there is an isomorphism of $k$-algebras
\[ 
\tExt_{G}^*(\gam_{\fp}(k),\gam_{\fp}(k)) \cong (H^{*}(G,k)_{\fp})^{\wedge}
\] where $(-)^{\wedge}$ denotes completion with respect to the
$\fp$-adic topology, and the $G$-module $\gam_{\fp}(k)$ is pure
injective.
\end{theorem}

Theorem~\ref{thm:main2} can be interpreted to mean that the category
$\StMod G$ is Gorenstein, for it is analogous to Grothendieck's result
that a commutative noetherian ring $A$ is Gorenstein if, and only if,
$\gam_{\fp}A$ is the injective hull of $A/\fp$, up to suspension, for
each $\fp$ in $\Spec A$. In Section~\ref{sec:gorenstein} we propose a
general notion of a Gorenstein triangulated category, in an attempt to
place these results in a common framework. There is a plethora of
duality results, and categorical frameworks dealing with them.  A
systematic comparison with those is beyond the scope of this work.

To prove Theorem~\ref{thm:main2} we use a technique from algebraic
geometry in the tradition of Zariski and Weil; namely, the
construction of generic points for algebraic varieties.  Given a point
$\fp\subseteq H^*(G,k)$, there is a purely transcendental extension
$K$ of $k$ and a closed point $\fm$ of $\Proj H^*(G_K,K)$ lying above
the point $\fp$ in $\Proj H^*(G,k)$. Here, $G_K$ denotes the group
scheme that is obtained from $G$ by extending the field to $K$.  The
crux is that \emph{one can choose} $\fm$ such that the following
statement holds.

\begin{theorem} 
Restriction of scalars induces a dense exact functor
\[ 
\stmod G_K\supseteq\tors_{\fm}(\stmod G_K)\longrightarrow \tors_{\fp}(\stmod G).
\]
\end{theorem}

This result is proved in Section~\ref{sec:closed-points}, building on
our work in \cite{Benson/Iyengar/Krause/Pevtsova:2015b}. It gives a
remarkable description of the compact objects in
$\gam_\fp(\StMod{G})$: they are obtained from the finite dimensional
objects in $\gam_\fm(\StMod G_K)$ by restriction of scalars. This
allows one to reduce the proof of Theorem~\ref{thm:main2} to the case
of a closed point, where it is essentially equivalent to classical
Tate duality.  The theorem above has other consequences; for example,
it implies that the compact objects in $\gam_{\fp}(\StMod G)$ are
endofinite $G$-modules in the sense of
Crawley-Boevey~\cite{Crawley-Boevey:1992a}; see
Section~\ref{sec:closed-points}.

\section{Cohomology and localisation}
\label{sec:cohomology-and-localisation} In this section we recall
basic notions concerning certain localisation functors on triangulated
categories with ring actions. The material is needed to state and
prove the results in this work. The main triangulated category of
interest is the stable module category of a finite group scheme, but
the general framework is needed in Sections~\ref{sec:gorenstein} and
\ref{sec:serre}.  Primary references for the material presented here
are \cite{Benson/Iyengar/Krause:2008a, Benson/Iyengar/Krause:2011a};
see \cite{Benson/Iyengar/Krause/Pevtsova:2015b} for the special case
of the stable module category.

\subsection*{Triangulated categories with central action} 
Let $\sfT$ be a triangulated category with suspension $\Si$. Given
objects $X$ and $Y$ in $\sfT$, consider the graded abelian groups
\[ \Hom_\sfT^*(X,Y)=\bigoplus_{i\in\bbZ}\Hom_\sfT(X,\Si^i Y) \quad
\text{and}\quad
\End_{\sfT}^{*}(X)= \Hom_{\sfT}^{*}(X,X)\,.
\] Composition makes $\End_{\sfT}^{*}(X)$ a graded ring and
$\Hom_{\sfT}^{*}(X,Y)$ a left-$\End_{\sfT}^{*}(Y)$
right-$\End_{\sfT}^{*}(X)$ module.

Let $R$ be a graded commutative noetherian ring.  In what follows we
will only be concerned with homogeneous elements and ideals in $R$. In
this spirit, `localisation' will mean homogeneous localisation, and
$\Spec R$ will denote the set of homogeneous prime ideals in $R$.

We say that a triangulated category $\sfT$ is \emph{$R$-linear} if for
each $X$ in $\sfT$ there is a homomorphism of graded rings
$\phi_X\colon R\to \End_{\sfT}^{*}(X)$ such that the induced left and
right actions of $R$ on $\Hom_{\sfT}^{*}(X,Y)$ are compatible in the
following sense: For any $r\in R$ and $\alpha\in\Hom^*_\sfT(X,Y)$, one
has
\[ \phi_Y(r)\alpha=(-1)^{|r||\alpha|}\alpha\phi_X(r)\,.
\]

An exact functor $F\colon \sfT\to\sfU$ between $R$-linear triangulated
categories is \emph{$R$-linear} if the induced map
\[ \Hom_\sfT^*(X,Y)\lra \Hom_\sfU^*(FX,FY)
\] of graded abelian groups is $R$-linear for all objects $X,Y$ in
$\sfT$.
 
In what follows, we fix a compactly generated $R$-linear triangulated
category $\sfT$ and write $\sfT^c$ for its full subcategory of compact
objects.

\subsection*{Localisation} 
Fix an ideal $\fa$ in $R$. An $R$-module $M$ is \emph{$\fa$-torsion}
if $M_\fq=0$ for all $\fq$ in $\Spec R$ with $\fa\not\subseteq
\fq$. Analogously, an object $X$ in $\sfT$ is \emph{$\fa$-torsion} if
the $R$-module $\Hom_\sfT^*(C,X)$ is $\fa$-torsion for all
$C\in\sfT^c$.  The full subcategory of $\fa$-torsion objects
\[ \gam_{\mcV(\fa)}\sfT:=\{X\in\sfT\mid X \text{ is $\fa$-torsion} \}
\] is localising and the inclusion $\gam_{\mcV(\fa)}\sfT\subseteq
\sfT$ admits a right adjoint, denoted $\gam_{\mcV(\fa)}$.

Fix a $\fp$ in $\Spec R$.  An $R$-module $M$ is \emph{$\fp$-local} if
the localisation map $M\to M_\fp$ is invertible, and an object $X$ in
$\sfT$ is \emph{$\fp$-local} if the $R$-module $\Hom_\sfT^*(C,X)$ is
$\fp$-local for all $C\in\sfT^c$.  Consider the full subcategory of
$\sfT$ of $\fp$-local objects
\[ \sfT_\fp:=\{X\in\sfT\mid X \text{ is $\fp$-local}\}
\] and the full subcategory of $\fp$-local and $\fp$-torsion objects
\[ \gam_\fp\sfT:=\{X\in\sfT\mid X \text{ is $\fp$-local and
$\fp$-torsion} \}.
\] Note that $\gam_\fp\sfT\subseteq\sfT_\fp\subseteq\sfT$ are
localising subcategories.  The inclusion $\sfT_\fp\to\sfT$ admits a
left adjoint $X\mapsto X_\fp$ while the inclusion
$\gam_\fp\sfT\to\sfT_\fp$ admits a right adjoint. We denote by
$\gam_\fp\colon\sfT\to\gam_\fp\sfT$ the composition of those adjoints;
it is the \emph{local cohomology functor} with respect to $\fp$; see
\cite{Benson/Iyengar/Krause:2008a, Benson/Iyengar/Krause:2011a} for
the construction of this functor.

The following observation is clear.

\begin{lemma}
\label{le:periodicity} For any element $r$ in $R\setminus \fp$, say of
degree $n$, and $\fp$-local object $X$, the natural map $X\xra{r}\Si^n
X$ is an isomorphism. \qed
\end{lemma}

The functor $\gam_{\mcV(\fa)}$ commutes with exact functors preserving
coproducts.

\begin{lemma}
\label{le:gamma-commute} Let $F\colon\sfT\to \sfU$ be an exact functor
between $R$-linear compactly generated triangulated categories such
that $F$ is $R$-linear and preserves coproducts. Suppose that the
action of $R$ on $\sfU$ factors through a homomorphism $f\colon R\to
S$ of graded commutative rings. For any ideal $\fa$ of $R$ there is a
natural isomorphism
\[ F\comp\gam_{\mcV(\fa)}\cong \gam_{\mcV(\fa S)} \comp F
\] of functors $\sfT\to\sfU$, where $\fa S$ denotes the ideal of $S$
that is generated by $f(\fa)$.
\end{lemma}

\begin{proof} 
  The statement follows from an explicit description of
  $\gam_{\mcV(\fa)}$ in terms of homotopy colimits; see
  \cite[Proposition~2.9]{Benson/Iyengar/Krause:2011a}.
\end{proof}

\subsection*{Injective cohomology objects}

Given an object $C$ in $\sfT^{\sfc}$ and an injective $R$-module $I$,
Brown representability yields an object $T(C,I)$ in $\sfT$ such that
\begin{equation}\label{eq:inj-coh} \Hom_R(\Hom^*_\sfT(C,-),
I)\cong\Hom_\sfT(-,T(C,I))\,.
\end{equation} This yields a functor
\[ T\colon \sfT^{\sfc}\times\Inj R\lra\sfT.
\] For each $\fp$ in $\Spec R$, we write $I(\fp)$ for the injective
hull of $R/\fp$ and set
\[ T_\fp:=T(-,I(\fp))\,,
\] viewed as a functor $\sfT^{\sfc}\to \sfT$.

\subsection*{Tensor triangulated categories}

Let now $\sfT=(\sfT,\otimes,\one)$ be a tensor triangulated category
such that $R$ acts on $\sfT$ via a homomorphism of graded rings
$R\to \End_{\sfT}^{*}(\one)$. It is easy to verify that in this case
the functors $\gam_\fp$ and $T_\fp$ are tensor functors:
\begin{equation}
\label{eq:tensor-identity} \gam_\fp\cong \gam_\fp(\one)\otimes
-\qquad\text{and}\qquad T_\fp\cong T_\fp(\one)\otimes -
\end{equation}

Now we turn to modules over finite group schemes.  We follow the
notation and terminology from
\cite{Benson/Iyengar/Krause/Pevtsova:2015b}.

\subsection*{The stable module category} 

Let $G$ be a finite group scheme over a field $k$ of positive
characteristic.  The coordinate ring and the group algebra of $G$ are
denoted $k[G]$ and $kG$, respectively. These are finite dimensional
Hopf algebras over $k$ that are dual to each other. We write $\Mod G$
for the category of $G$-modules and $\mod G$ for its full subcategory
consisting of finite dimensional $G$-modules. We often identity
$\Mod G$ with the category of $kG$-modules, which is justified by
\cite[I.8.6]{Jantzen:2003a}.

We write $H^*(G,k)$ for the cohomology algebra, $\Ext^{*}_{G}(k,k)$,
of $G$. This is a graded commutative $k$-algebra, because $kG$ is a
Hopf algebra, and acts on $\Ext^{*}_{G}(M,N)$, for any $G$-modules
$M,N$. Moreover, the $k$-algebra $H^{*}(G,k)$ is finitely generated,
and, when $M,N$ are finite dimensional, $\Ext^{*}_{G}(M,N)$ is
finitely generated over it; this is by a theorem of Friedlander and
Suslin~\cite{Friedlander/Suslin:1997a}.

The stable module category $\StMod G$ is obtained from $\Mod G$ by
identifying two morphisms between $G$-modules when they factor through
a projective $G$-module.  An isomorphism in $\StMod G$ will be called
a \emph{stable isomorphism}, to distinguish it from an isomorphism in
$\Mod G$. In the same vein, $G$-modules $M$ and $N$ are said to be
\emph{stably isomorphic} if they are isomorphic in $\StMod G$; this is
equivalent to the condition that $M$ are $N$ are isomorphic in $\Mod
G$, up to projective summands.

The tensor product over $k$ of $G$-modules passes to $\StMod G$ and
yields a tensor triangulated category with unit $k$ and suspension
$\Omega^{-1}$, the inverse of the syzygy functor. The category $\StMod
G$ is compactly generated and the subcategory of compact objects
identifies with $\stmod G$, the stable module category of finite
dimensional $G$-modules. See Carlson \cite[\S5]{Carlson:1996a} and
Happel~\cite[Chapter I]{Happel:1988a} for details.

We use the notation $\sHom_G(M,N)$ for the $\Hom$-sets in $\StMod G$.
The cohomology algebra $H^{*}(G,k)$ acts on $\StMod G$ via a
homomorphism of $k$-algebras
\[ -\otimes_k M\colon H^*(G,k)=\Ext_G^*(k,k)\lra\sHom^{*}_{G}(M,M)\,.
\] Thus, the preceding discussion on localisation functors on
triangulated categories applies to the $H^{*}(G,k)$-linear category
$\StMod G$.

\subsection*{Koszul objects} 

An element $b$ in $H^{d}(G,k)$ corresponds to a morphism
$k\to \Omega^{-d}k$ in $\StMod G$; let $\kos k{b}$ denote its mapping
cone. This gives a morphism $k\to \Omega^d(\kos k{b})$. For a sequence
of elements $\bsb:=b_{1},\dots,b_{n}$ in $H^{*}(G,k)$ and a $G$-module
$M$, we set
\[ \kos k{\bsb}:=(\kos k{b_1})\otimes_k\cdots\otimes_k (\kos k{b_n})
\qquad\text{and}\qquad \kos M{\bsb}:=M\otimes_k \kos k{\bsb}.
\] It is easy to check that for a $G$-module $N$ and
$s=\sum_{i}|b_{i}|$, there is an isomorphism
\begin{equation}
\label{eq:kos-swap} \sHom_{G}(M,\kos N{\bsb}) \cong
\sHom_{G}(\Omega^{n+s}\kos M{\bsb},N).
\end{equation} 
Let $\fb=(\bsb)$ denote the ideal of $H^{*}(G,k)$ generated by
$\bsb$. By abuse of notation we set $\kos M{\fb}:= \kos M {\bsb}$. If
$\bsb'$ is a finite set of elements in $H^{*}(G,k)$ such that
$\sqrt{(\bsb')}=\sqrt{(\bsb)}$, then, by
\cite[Proposition~3.10]{Benson/Iyengar/Krause:2015a}, for any $M$ in
$\StMod G$ there is an equality
\begin{equation}
\label{eq:kos-thick} \Thick (\kos M{\bsb})=\Thick (\kos M{\bsb'})\,.
\end{equation}

Fix $\fp$ in $\Spec H^{*}(G,k)$. We will repeatedly use the fact that
$\gam_\fp(\StMod G)^c$ is generated as a triangulated category by the
family of objects $(\kos M{\fp})_\fp$ with $M$ in $\stmod G$; see
\cite[Proposition~3.9]{Benson/Iyengar/Krause:2011a}. In fact, if $S$
denotes the direct sum of a representative set of simple $G$-modules,
then there is an equality
\begin{equation}
\label{eq:generators} \gam_\fp(\StMod G)^c=\Thick((\kos
S{\fp})_\fp)\,.
\end{equation} 
It turns out that one has $\gam_{\fm}(\StMod G)=\{0\}$ where $\fm$
denotes $H^{\ges 1}(G,k)$, the ideal of elements of positive degree;
see Lemma~\ref{le:not-m} below. For this reason, it is customary to
focus on $\Proj H^{*}(G,k)$, the set of homogeneous prime ideals not
containing $\fm$, when dealing with $\StMod G$.

\subsection*{Tate cohomology} 

By construction, the action of $H^{*}(G,k)$ on $\StMod G$ factors
through $\sHom^{*}_{G}(k,k)$, the graded ring of endomorphisms of the
identity. The latter ring is not noetherian in general, which is one
reason to work with $H^{*}(G,k)$. In any case, there is little
difference, vis a vis their action on $\StMod G$, as the next remarks
should make clear.

\begin{remark}
\label{re:Tate-coh} Let $M$ and $N$ be $G$-modules. The map
$\Hom_{G}(M,N)\to \sHom_{G}(M,N)$ induces a map
$\Ext^{*}_{G}(M,N)\to\sHom^{*}_{G}(M,N)$ of $H^*(G,k)$-modules. This
map is surjective in degree zero, with kernel $\PHom_{G}(M,N)$, the
maps from $M$ to $N$ that factor through a projective $G$-module. It
is bijective in positive degrees and hence one gets an exact sequence
of graded $H^{*}(G,k)$-modules
\begin{equation}
\label{eq:Tate-coh} 0\lra \PHom_{G}(M,N) \lra \Ext^{*}_{G}(M,N)\lra
\sHom^{*}_{G}(M,N) \lra X\lra 0
\end{equation} 
with $X^{i}=0$ for $i\ge 0$. For degree reasons, the
$H^{*}(G,k)$-modules $\PHom_{G}(M,N)$ and $X$ are
$\fm$-torsion. Consequently, for $\fp$ in $\Proj H^{*}(G,k)$ the
induced localised map is an isomorphism:
\begin{equation}
\label{eq:Tate-local} \Ext^{*}_{G}(M,N)_{\fp}\lra
\sHom^{*}_{G}(M,N)_{\fp}\,.
\end{equation} 
More generally, for each $r$ in $\fm$ localisation induces an
isomorphism
\[ \Ext^{*}_{G}(M,N)_{r}\xra{\ \cong\ } \sHom^{*}_{G}(M,N)_{r}
\] of $H^{*}(G,k)_{r}$-modules. This means that $\Proj H^{*}(G,k)$ has
a finite cover by affine open sets on which ordinary cohomology and
stable cohomology agree.
\end{remark}

Given the finite generation result due to Friedlander and Suslin
mentioned earlier, the next remark can be deduced from the exact
sequence \eqref{eq:Tate-coh}.

\begin{remark}
\label{re:fin-gen} When $M,N$ are finite dimensional $G$-modules,
$\sHom_{G}^{\ges s}(M,N)$ is a finitely generated $H^{*}(G,k)$-module
for any $s\in\bbZ$. Moreover the $H^{*}(G,k)_{\fp}$-module
\[ \sHom_{G}^{*}(M_{\fp},N_{\fp})\cong \sHom_{G}^{*}(M,N)_{\fp}
\] is finitely generated for each $\fp$ in $\Proj H^{*}(G,k)$.
\end{remark}

\begin{lemma}
\label{le:not-m} One has $\gam_{\fm}(\StMod G)=\{0\}$, where
$\fm=H^{\ges 1}(G,k)$.
\end{lemma}

\begin{proof} 
  Given \eqref{eq:generators} it suffices to check that
  $\kos S{\fm}=0$ in $\StMod G$, where $S$ is the direct sum of
  representative set of simple $G$-modules. For any $G$-module $M$,
  the $H^{*}(G,k)$-module $\sHom_{G}^{*}(M,\kos S\fm)$ is
  $\fm$-torsion; see
  \cite[Lemma~5.11(1)]{Benson/Iyengar/Krause:2008a}. Thus, when $M$ is
  finite dimensional, the $H^{*}(G,k)$-module
  $\sHom_{G}^{\ges 0}(M,\kos S\fm)$ is $\fm$-torsion and finitely
  generated, by Remark~\ref{re:fin-gen}, so it follows that
  $\sHom_{G}^{i}(M,\kos S\fm)=0$ for $i\gg 0$. This implies that
  $\kos S\fm$ is projective, since $kG$ is self-injective.
\end{proof}

To gain a better understanding of the discussion above, it helps to
consider the homotopy category of $\Inj G$, the injective $G$-modules.

\subsection*{The homotopy category of injectives} 

Let $\KInj G$ and $\sfD(\Mod G)$ denote the homotopy category of
$\Inj G$ and the derived category of $\Mod G$, respectively. These are
also $H^*(G,k)$-linear compactly generated tensor triangulated
category, with the tensor product over $k$. The unit of the tensor
product on $\KInj G$ is an injective resolution of the trivial
$G$-module $k$, while that of $\sfD(\Mod G)$ is $k$. The canonical
quotient functor $\KInj G\to \sfD(\Mod G)$ induces an equivalence of
triangulated category ${\KInj G}^c\xra{\sim}\sfD^b(\mod G)$, where the
target is the bounded derived category of $\mod G$; see
\cite[Proposition~2.3]{Krause:2005a}.

Taking Tate resolutions identifies $\StMod G$ with $\KacInj G$, the
full subcategory of acyclic complexes in $\KInj G$. In detail, let
$\sfp k$ and $\sfi k$ be a projective and an injective resolution of
the trivial $G$-module $k$, respectively, and let $\sft k$ be the
mapping cone of the composed morphism $\sfp k \to k\to \sfi k$; this
is a Tate resolution of $k$. Since projective and injective
$G$-modules coincide, one gets the exact triangle
\begin{equation}
\label{eq:tate-triangle} \sfp k\lra \sfi k \lra \sft k\lra
\end{equation} in $\KInj G$. For a $G$-module $M$, the complex
$M\otimes_{k}\sft k$ is a Tate resolution of $M$ and the assignment
$M\mapsto M\otimes_{k}\sft k$ induces an equivalence of categories
\[ \StMod G \xra{\sim} \KacInj G,
\] with quasi-inverse $X\mapsto Z^{0}(X)$, the submodule of cycles in
degree $0$. Assigning $X$ in $\KInj G$ to $X\otimes_{k}\sft k$ is a
left adjoint of the inclusion $\KacInj G\to \KInj G$. These results
are contained in \cite[Theorem~8.2]{Krause:2005a}.  Consider the
composed functor
\[ \pi\colon \KInj G\xra{ -\otimes_{k}\sft k} \KacInj G\xra{\ \sim\ }
\StMod G\,.
\] A straightforward verification yields that these functors are
$H^{*}(G,k)$-linear. The result below is the categorical underpinning
of Remark~\ref{re:Tate-coh} and Lemma~\ref{le:not-m}.

\begin{lemma}
\label{le:KInj-p-local} There is a natural isomorphism $\gam_{\fm}X
\cong X\otimes_{k}\sfp k$ for $X\in \KInj G$. For each $\fp$ in $\Proj
H^*(G,k)$, the functor $\pi$ induces triangle equivalences
  \[ {\KInj G}_\fp\xra{\sim}(\StMod G)_\fp\qquad\text{and}\qquad
\gam_\fp(\KInj G )\xra{\sim}\gam_\fp(\StMod G).
  \]
\end{lemma}

\begin{proof} 
  We identify $\StMod G$ with $\KacInj G$. This entails
  $\gam_{\fm}(\KacInj G)=\{0\}$, by Lemma~\ref{le:not-m}. It is easy
  to check that $kG$ is $\fm$-torsion, and hence so is $\sfp k$, for
  it is in the localising subcategory generated by $kG$, and the class
  of $\fm$-torsion objects in $\KInj G$ is a tensor ideal localising
  subcategory; see, for instance, \cite[Section
  8]{Benson/Iyengar/Krause:2008a}. Thus, applying $\gam_{\fm}(-)$ to
  the exact triangle \eqref{eq:tate-triangle} yields
  $\sfp k \cong \gam_{\fm}(\sfi k)$. It then follows from
  \eqref{eq:tensor-identity} that
  $X\otimes_{k}\sfp k\cong \gam_{\fm}X$ for any $X$ in $\KInj G$.

From the construction of $\pi$ and \eqref{eq:tate-triangle}, the
kernel of $\pi$ is the subcategory
\[ \{X\in \KInj G \mid X\otimes_{k}\sfp k\cong X\}.
\] These are precisely the $\fm$-torsion objects in $\KInj G$, by the
already established part of the result. Said otherwise, $X\in \KInj G$
is acyclic if and only if $\gam_{\fm}X=0$.  It follows that $\KacInj
G$ contains the subcategory ${\KInj G}_\fp$ of $\fp$-local objects,
for each $\fp$ in $\Proj H^{*}(G,k)$. On the other hand, the inclusion
$\KacInj G\subseteq \KInj G$ preserves coproducts, so its left adjoint
$\pi$ preserves compactness of objects and all compacts of $\KacInj G$
are in the image of $\pi$. Given this a simple calculation shows that
${\KInj G}_\fp$ contains ${\KacInj G}_\fp$. Thus ${\KacInj G}_\fp=
{\KInj G}_\fp$.
\end{proof}

\section{Passage to closed points}
\label{sec:closed-points} Let $G$ be a finite group scheme over a
field $k$ of positive characteristic.  In this section we describe a
technique that relates the $\fp$-local $\fp$-torsion objects in
$\StMod G$, for a point $\fp$ in $\Proj H^{*}(G,k)$, to the
corresponding modules at a closed point defined over a field extension
of $k$. Recall that a point $\fm$ is \emph{closed} when it is maximal
with respect to inclusion: $\fm\subseteq \fq$ implies $\fm=\fq$ for
all $\fq$ in $\Proj H^{*}(G,k)$. In what follows, $k(\fp)$ denotes the
graded residue field of $H^{*}(G,k)$ at $\fp$.

For a field extension $K/k$ extension of scalars and restriction give
exact functors
\[ K\otimes_{k}(-)\colon \StMod G\lra \StMod G_{K}
\quad\text{and}\quad (-)\da_{G}\colon \StMod G_K\lra \StMod G\,.
\] Moreover, since $H^{*}(G_{K},K)\cong K\otimes_{k} H^{*}(G,k)$ as
$K$-algebras, $K\otimes_{k}(-)$ yields a homomorphism $H^{*}(G,k)\to
H^{*}(G_{K},K)$ of rings. It induces a map
\[ 
\Proj H^{*}(G_{K},K) \lra \Proj H^{*}(G,k)\,,
\] 
with $\fq$ mapping to $\fp:=\fq\cap H^{*}(G,k)$. We say that $\fq$ \emph{lies over} $\fp$ to indicate this. 
The main objective of this section is the proof of the following result.

\begin{theorem}
\label{th:realisability} 
Fix $\fp$ in $\Proj H^{*}(G,k)$ and $K/k$ a purely transcendental extension of degree $\dim (H^{*}(G,k)/\fp)-1$. There exists a closed point $\fm$ in $\Proj H^{*}(G_{K},K)$ lying over $\fp$ with $k(\fm)\cong k(\fp)$ such that the functor $(-)\da_{G}$ restricts to functors
\[ 
\gam_\fm(\StMod G_K)\to \gam_\fp(\StMod G) \quad\text{and}\quad
\gam_\fm(\StMod G_K)^c\to \gam_\fp(\StMod G)^c
\] 
that are dense.
\end{theorem}

The proof of the theorem yields more: There is a subcategory of $\gam_\fm(\StMod G_K)$ on which $(-)\da_{G}$ is full and dense; ditto for the category of compact objects. However, the functor need not be full on all of $\gam_\fm(\StMod G_K)^{c}$; see Example~\ref{ex:klein-again}.

Here is one consequence of Theorem~\ref{th:realisability}.

\begin{corollary}
\label{co:realisability} The compact objects in $\gam_{\fp}(\StMod G)$
are precisely the restrictions of finite dimensional $G_{K}$-modules
in $\gam_{\fm}(\StMod G_K)$.
\end{corollary}

\begin{proof} 
By \cite[Theorem~6.4]{Benson/Iyengar/Krause:2008a}, for any ideal $\fa$ in $H^{*}(G,k)$, we have
\[ 
\gam_{\mcV(\fa)}(\StMod G)^c=\gam_{\mcV(\fa)}(\StMod G)\cap\stmod G\,.
\] 
Applying this observation to the ideal $\fm$ of $H^{*}(G_{K},K)$
and noting that $\gam_{\fm}=\gam_{\mcV(\fm)}$, since $\fm$ is a closed
point, the desired result follows from Theorem~\ref{th:realisability}.
\end{proof}

The closed point in Theorem~\ref{th:realisability} depends on the
choice of a Noether normalisation of $H^{*}(G,k)/\fp$ as is explained
in the construction below, from
\cite[\S7]{Benson/Iyengar/Krause/Pevtsova:2015b}.

\begin{construction}
\label{con:generic} 
Fix $\fp$ in $\Proj H^{*}(G,k)$; the following
construction is relevant only when $\fp$ is not a closed point. Choose
elements $\bsa:=a_{0},\dots,a_{d-1}$ in $H^{*}(G,k)$ of the same
degree such that their image in $H^{*}(G,k)/\fp$ is algebraically
independent over $k$ and $H^{*}(G,k)/\fp$ is finitely generated as a
module over the subalgebra $k[\bsa]$.  Thus the Krull dimension of
$H^{*}(G,k)/\fp$ is $d$. Set $K:=k(t_{1},\dots,t_{d-1})$, the field of
rational functions in indeterminates $t_{1},\dots,t_{d-1}$ and
\[ 
b_{i}:= a_{i} - a_{0}t_{i}\quad\text{for $i=1,\dots,d-1$}
\] 
viewed as elements in $H^{*}(G_{K},K)$. Let $\fp'$ denote the
extension of $\fp$ to $H^*(G_K,K)$, and set
\begin{equation*}
\label{eq:generic-point} \fq:= \fp' + (\bsb)\qquad\text{and}\qquad
\fm:=\sqrt \fq\,.
\end{equation*} 
It is proved as part of
\cite[Theorem~7.7]{Benson/Iyengar/Krause/Pevtsova:2015b} that the
ideal $\fm$ is a closed point in $\Proj H^{*}(G_{K},K)$ with the
property that $\fm\cap H^{*}(G,k)=\fp$. What is more, it follows from
the construction (see in particular \cite[Lemma~7.6, and
(7.2)]{Benson/Iyengar/Krause/Pevtsova:2015b}) that the induced
extension of fields is an isomorphism
\begin{equation*}
\label{eq:residue} k(\fp)\xra{\ \cong\ } k(\fm)\,.
 \end{equation*}

The sequence of elements $\bsb$ in $H^{*}(G_{K},K)$ yields a morphism
$K\to \Omega^{s} (\kos K{\bsb})$, where $s=\sum_{i}|b_{i}|$, and
composing its restriction to $G$ with the canonical morphism $k\to
K\da_{G}$ gives in $\StMod G$ a morphism
\[ f\colon k \lra \Omega^{s} (\kos K{\bsb})\da_{G}\,.
\] Since the $a_{i}$ are not in $\fp$, Lemma~\ref{le:periodicity}
yields a natural stable isomorphism
\begin{equation}
\label{eq:periodicity} \Omega^{s}M \cong M
\end{equation} for any $\fp$-local $G$-module $M$. This remark will be
used often in the sequel.
\end{construction}

By \cite[Lemma~2.1]{Benson/Iyengar/Krause/Pevtsova:2016a}, for any
$G_K$-module $N$ there is a natural isomorphism
\begin{equation}\label{eq:res-ind} M\otimes_k N\da_G\cong
(M_K\otimes_K N)\da_G.
\end{equation}

The result below extends
\cite[Theorem~8.8]{Benson/Iyengar/Krause/Pevtsova:2015b}; the latter
is the case $M=\kos k{\fp}$.

\begin{theorem}
\label{thm:realisability} For any $G$-module $M$, the morphism
$M\otimes_k f$ induces a natural stable isomorphism of $G$-modules
\[ \gam_{\fp}M\cong M\otimes_k \gam_\fm(\kos K{\bsb})\da_{G} \cong
(M_K\otimes_K \gam_\fm(\kos K{\bsb}))\da_{G}\,.
\] When $M$ is $\fp$-torsion, these induce natural stable isomorphisms
\[ \gam_{\fp}M \cong M_\fp\cong M\otimes_k (\kos K{\bsb})\da_{G} \cong
(M_K\otimes_K \kos K{\bsb})\da_{G} \,.
\]
\end{theorem}

\begin{proof} 
  We begin by verifying the second set of isomorphisms. As $M$ is
  $\fp$-torsion so is $M_{\fp}$ and then it is clear that the natural
  map $\gam_{\fp}M=\gam_{\mcV(\fp)}M_\fp\to M_{\fp}$ is an
  isomorphism. The third of the desired isomorphisms follows from
  \eqref{eq:res-ind}. It thus remains to check that $M\otimes_{k}f$
  induces an isomorphism
\[ M_\fp\cong M\otimes_k (\kos K{\bsb})\da_{G}\,.
\] It is easy to verify that the modules $M$ having this property form
a tensor ideal localising subcategory of $\StMod G$.  Keeping in mind
~\eqref{eq:periodicity}, from
\cite[Theorem~8.8]{Benson/Iyengar/Krause/Pevtsova:2015b} one obtains
that this subcategory contains $\kos k{\fp}$. The desired assertion
follows since the $\fp$-torsion modules form a tensor ideal localising
subcategory of $\StMod G$ that is generated by $\kos k{\fp}$; see
\cite[Proposition~2.7]{Benson/Iyengar/Krause:2011a}.

Now we turn to the first set of isomorphisms. There the second one is
by \eqref{eq:res-ind}, so we focus on the first. Let $M$ be an
arbitrary $G$-module, and let $\fp'$ be as in
Construction~\ref{con:generic}. Since $\gam_{\mcV(\fp)}M$ is
$\fp$-torsion, the already established isomorphism yields the second
one below.
\begin{align*} 
\gam_{\fp}M&\cong (\gam_{\mcV(\fp)} M)_\fp\\ & \cong
((\gam_{\mcV(\fp)} M)_K\otimes_K \kos K{\bsb})\da_{G}\\ & \cong
(\gam_{\mcV(\fp')} (M_K)\otimes_K \kos K{\bsb})\da_{G}\\ & \cong
(M_K\otimes_K \gam_{\mcV(\fp')}(\kos K{\bsb}))\da_{G}\\ & \cong
(M_K\otimes_K \gam_{\mcV(\fp'+(\bsb))}(\kos K{\bsb}))\da_{G}\\ & \cong
(M_K\otimes_K \gam_\fm(\kos K{\bsb}))\da_{G}
\end{align*} 
The third one is by Lemma~\ref{le:gamma-commute}, applied
to the functor $K\otimes_{k}(-)$ from $\StMod G$ to $\StMod
G_{K}$. The next one is standard while the penultimate one holds
because $\kos K{\bsb}$ is $(\bsb)$-torsion. This completes the proof.
\end{proof}

In the next remark we recast part of Theorem~\ref{thm:realisability}.

\begin{remark}
\label{rem:adjoints} 
Fix a point $\fp$ in $\Proj H^{*}(G,k)$, and let
$K$, $\bsb$ and $\fm$ be as in Construction~\ref{con:generic}.
Consider the following adjoint pair of functors.
\begin{alignat*}{3} \lambda\colon &\StMod G \lra \StMod G_{K}&
\quad\text{and}\quad \rho\colon &\StMod G_{K} \lra \StMod G \\
&\lambda(M) = M_{K}\otimes_{K}\kos K{\bsb} & &\rho(N) = \Hom_{K}(\kos
K{\bsb},N)\da_{G}
\end{alignat*} 
It is easy to check that this induces an adjoint pair
\[
\begin{tikzcd} 
  \gam_{\fp}(\StMod G) \arrow[yshift=1ex]{rr}{\lambda} &&
  \gam_{\fm}(\StMod G_{K}) \arrow[yshift=-1ex]{ll}{\rho}\, .
\end{tikzcd}
\] 
Theorem~\ref{thm:realisability} implies that $(\lambda M)\da_{G}\cong M$ for any $M$ in $\gam_{\fp}(\StMod G)$. 
\end{remark}

\begin{proof}[Proof of Theorem~\ref{th:realisability}] 
 Let $\fm$, $\fq$, and $\bsb$ be as in
  Construction~\ref{con:generic}. As noted there, $\fm$ is a closed
  point in $\Proj H^{*}(G_{K},K)$ lying over $\fp$ and
  $k(\fm)\cong k(\fp)$.  The modules in $\gam_\fp(\StMod G)$ are
  precisely those with support contained in $\{\fp\}$. It then follows
  from \cite[Proposition~6.2]{Benson/Iyengar/Krause/Pevtsova:2015b}
  that $(-)\da_G$ restricts to a functor
\[ 
\gam_\fm(\StMod G_K)\lra \gam_\fp(\StMod G)\,.
\] 
This functor is  dense because for any $G$-module $M$ that is $\fp$-local and $\fp$-torsion one has $M\cong (\lambda M)\da_{G}$ where $\lambda$ is the functor from Remark~\ref{rem:adjoints}.

Consider the restriction of $(-)\da_G$ to compact objects in $\gam_{\fm}(\StMod G_{K})$.  First we verify that its image is contained in the compact objects of $\gam_\fp(\StMod G)$. To this end, it suffices to check that there exists a generator of $\gam_\fm(\StMod
G_K)^{\sfc}$, as a thick subcategory, whose restriction is in $\gam_\fp(\StMod G)^{\sfc}$.

Let $S$ be the direct sum of a representative set of simple $G$-modules. Each simple $G_{K}$-module is (isomorphic to) a direct summand of $S_{K}$, so from \eqref{eq:generators} one gets the first equality below:
\[ 
\gam_\fm(\StMod G_K)^c=\Thick(\kos {S_K}{\fm})=\Thick(\kos{S_{K}}{\fq})\,.
\] 
The second one is by \eqref{eq:kos-thick}. From Theorem~\ref{thm:realisability} one gets isomorphisms of $G$-modules
\[ 
(\kos {S_K}{\fq})\da_G \cong ((\kos S{\fp})_K\otimes_K \kos K{\bsb})\da_G\cong (\kos S{\fp})_\fp\,.
\] 
It remains to note that $(\kos S{\fp})_\fp$ is in $\gam_{\fp}(\StMod G)^{\sfc}$, again by \eqref{eq:generators}.

The last item to verify is that restriction is dense also on compacts. Since $\kos K{\bsb}$ is compact, the  functor $\rho$ from Remark~\ref{rem:adjoints} preserves coproducts, and hence its left adjoint $\lambda$ preserves compactness. Thus Theorem~\ref{thm:realisability} gives the desired result.
\end{proof}

Theorem~\ref{thm:realisability} yields that $f\cong (f_{K})\da_{G}$ for any morphism $f$ in $\gam_{\fp}(\StMod G)$; in particular, the restriction functor is full and dense on the subcategory of $\gam_{\fm}(\StMod G_{K})$ consisting of objects of the form $\lambda M$, where $M$ is a $\fp$-local $\fp$-torsion $G$-module. It need not be full on the entire category, or even on its subcategory of compact objects; see Example~\ref{ex:klein-again}, modeled on the following one from commutative algebra. 

\begin{example}
\label{ex:dvr}
Let $k$ be a field and $k[a]$ the polynomial ring in an indeterminate $a$. Let $\sfD(k[a])$ denote its derived category; it is $k[a]$-linear in an obvious way. For the prime  $\fp:=(0)$ of $k[a]$ the $\fp$-local $\fp$-torsion subcategory $\gam_{\fp}(\sfD(k[a]))$ is naturally identified with the derived category of $k(a)$, the field of rational functions in $a$.

With $k(t)$ denoting the field of rational functions in an indeterminate $t$, the maximal ideal $\fm:=(a-t)$ of $k(t)[a]$ lies over the prime ideal $\fp$ of $k[a]$. The inclusion $k[a]\subset k(t)[a]$ induces an isomorphism $k(a)\cong k(t)[a]/\fm\cong k(t)$. The analogue of Theorem~\ref{thm:realisability} is that restriction of scalars along the inclusion $k[a]\subset k(t)[a]$ induces a dense functor
\[
\gam_{\fm}(\sfD(k(t)[a]))\lra \gam_{\fp}(\sfD(k[a])) \simeq \sfD(k(a))\,.
\]
This property can be checked directly: The $\fm$-torsion module $k(t)[a]/(a-t)$ restricts to $k(a)$, and each object in $\sfD(k(a))$ is a direct sum of shifts of $k(a)$. This functor is however not full: For $n\ge 1$, the $k(t)[a]$-module $L:=k(t)[a]/(a-t)^{n}$ is $\fm$-torsion, and satisfies
\[
\rank_{k(a)}\End_{\sfD}(L) = n \quad\text{and}\quad  \rank_{k(a)}\End_{\sfD}(L\da_{k[a]}) = n^{2}
\]
where $\sfD$ stands for the appropriate derived category. In particular, if $n\ge 2$, the canonical map 
$\End_{\sfD}(L)\to \End_{\sfD}(L\da_{k[a]})$ is not surjective. 

Indeed, the module of endomorphisms of $L$ in  $\sfD(k(t)[a])$ is
\[
\End_{\sfD}(L) = \Hom_{k(t)[a]}(L,L) \cong L\,.
\]
In particular, it has rank $n$ as an $k(a)$-vector space. On the other hand, restricted to $k[a]$, the  $k(t)[a]$-module $k(t)/(a-t)$ is isomorphic to $k(a)$. It then follows from the exact sequences
\[
0\lra \frac {k(t)[a]}{(a-t)} \xra{1\mapsto (a-t)^{i}}  \frac {k(t)[a]}{(a-t)^{i+1}} \lra \frac {k(t)[a]}{(a-t)^{i}}\lra 0
\]
of $k(t)[a]$-modules that $L$ restricts to a direct sum of $n$ copies of $k(a)$, so that 
\[
\End_{\sfD}(L\da_{k[a]}) = \Hom_{k(a)}(k(a)^{n},k(a)^{n})\cong k(a)^{n^{2}}\,.
\]
In particular, this has rank $n^{2}$ as a $k(a)$-vector space.
\end{example}

\begin{example}
\label{ex:klein-again}
Let $V =  \bbZ/2 \times \bbZ/2$ and $k$ a field of characteristic two. As $k$-algebras, one has $H^{*}(V,k)\cong k[a,b]$, where $a$ and $b$ are indeterminates of degree one. For the prime ideal $\fp=(0)$ of $k[a,b]$, Construction~\ref{con:generic} leads to the field extension $K:=k(t)$ of $k$, and the closed point $\fm = (b-at)$ of $\Proj H^{*}(V_{K},K)$. 

Set $F:=\sEnd_{V}(k_{\fp})$; this is the component in degree $0$ of the graded field $k[a,b]_{\fp}$ and can be identified with $K$; see Construction~\ref{con:generic}.

Fix an integer $n\ge 1$ and set $N:=\kos K{(b-at)^{n}}$. This is a finite-dimensional $\fm$-torsion $V_{K}$-module and hence compact in $\gam_{\fm}(\StMod V_{K})$. We claim that
\[
 \rank_{F}\sEnd_{V_{K}}(N) = 2n \quad\text{and}\quad \rank_{F}\sEnd_{V}(N\da_{V}) = n^{2}\,,
\]
and hence that the map $\sEnd_{V_{K}}(N)\to \sEnd_{V}(N\da_{V})$ is not surjective when $n\ge 3$.

The claim can be checked as follows: Set $S:=\sEnd^{*}_{V_{K}}(K)_{\fm}\cong K[a,b]_{\fm}$. Since $(b-at)^{n}$ is not a zerodivisor on $S$,  applying $\sHom_{V_{K}}(K,-)$ to the exact triangle
\[
K \xra{ (b-at)^{n}} \Omega^{-n} K \lra N \lra
\]
one gets that $\sHom^{*}_{V_{K}}(K,N)$ is isomorphic to $S/(b-at)^{n}$, as an $S$-module; in particular $(b-at)^{n}$ annihilates it. Given this, applying $\sHom_{V_{K}}(-,N)$ to the exact triangle above yields that the rank of $\sEnd_{V_{K}}(N)$, as an $F$-vector space, is $n$.

As to the claim about $N\da_{V}$: the category $\gam_{\fp}(\StMod V)$ is semisimple for its generator $k_{\fp}$ has the property that $\sEnd_{V}^{*}(k_{\fp})$ is a graded field. It thus suffices to verify that $N\da_{V}\cong k_{\fp}^{n}$; equivalently, that 
$\rank_{F}\sHom_{V}(k_{\fp},N\da_{V})=n$. This follows from the isomorphisms
\[
\sHom_{V}(k_{\fp},N\da_{V}) \cong \sHom_{V}(k,N\da_{V})\cong \sHom_{V_{K}}(K,N) \cong F^{n}\,.
\]
The first one holds because $N\da_{V}$ is $\fp$-local, the second one is by adjunction.

There is a close connection between this example and Example~\ref{ex:dvr}. Namely,  the Bernstein-Gelfand-Gelfand correspondence sets up an equivalence between $\StMod V$ and the derived category of dg modules over $R:=k[a,b]$, viewed as a dg algebra with zero differential, modulo the subcategory of $(a,b)$-torsion dg modules; see, for example, \cite[\S{5.2.2}]{Benson/Iyengar/Krause:2012a}. The BGG correspondence induces the equivalences in the following commutative diagram
of categories.
\[
\begin{tikzcd}
\gam_{\fm}(\StMod V_{K})  \arrow[d, swap, "(-)\da_{V}"]  & \arrow[l,swap,"\simeq"] \gam_{\fm}(\sfD(S)) \arrow[d] \\
\gam_{\fp}(\StMod V)   & \arrow[l,swap,"\simeq"] \sfD(R_{\fp}) 
\end{tikzcd}
\]
where $\sfD(-)$ denotes the derived category of dg modules. The functor on the right is restriction of scalars along the homomorphism of rings $R_{\fp}\to S$, which is induced by the inclusion $R=k[a,b]\subset K[a,b]$. Under the BGG equivalence, the $V_{K}$-module $N$ corresponds to $S/(b-at)^{n}$, viewed as dg $S$-module with zero differential. Since $R_{\fp}$ is a graded field, isomorphic to $K[a^{\pm 1}]$, each dg $R_{\fp}$-module is isomorphic to a direct sum of copies of $R_{\fp}$.  Arguing as in Example~\ref{ex:dvr} one can verify that the dg $S$-module $S/(b-at)^{n}$ restricts to a direct sum of $n$ copies of $R_{\fp}$. This is another way to compute the endomorphism rings in question.
\end{example}

The remainder of this section is devoted to a further discussion of the compact objects in $\gam_{\fp}(\StMod G)$. This is not needed in
the sequel.

\subsection*{Endofiniteness}

Following Crawley-Boevey \cite{Crawley-Boevey:1991a,
Crawley-Boevey:1992a}, a module $X$ over an associative ring $A$ is
\emph{endofinite} if $X$ has finite length as a module over
$\End_A(X)$.

An object $X$ of a compactly generated triangulated category $\sfT$ is
\emph{endofinite} if the $\End_\sfT(X)$-module $\Hom_\sfT(C,X)$ has
finite length for all $C\in\sfT^c$; see
\cite{Krause/Reichenbach:2000a}.

Let $A$ be a self-injective algebra, finite dimensional over some
field. Then an $A$-module is endofinite if and only if it is
endofinite as an object of $\StMod A$. This follows from the fact $X$
is an endofinite $A$-module if and only if the $\End_A(X)$-module
$\Hom_A(C,X)$ has finite length for every finite dimensional
$A$-module $C$.

\begin{lemma}
\label{le:endofinite} Let $F\colon \sfT\to\sfU$ be a functor between
compactly generated triangulated categories that preserves products
and coproducts.  Let $X$ be an object in $\sfT$. If $X$ is endofinite,
then so is $FX$ and the converse holds when $F$ is fully faithful.
\end{lemma}

\begin{proof} 
By Brown representability, $F$ has a left adjoint, say $F'$.  It preserves compactness, as $F$ preserves coproducts. For $X\in\sfT$ and $C\in\sfU^c$, there is an isomorphism
\[ 
\Hom_\sfU(C,FX)\cong\Hom_\sfT(F'C,X)
\] 
of $\End_{\sfT}(X)$-modules. Thus if $X$ is endofinite, then $\Hom_\sfU(C,FX)$ is a module of finite length over $\End_\sfT(X)$, and therefore also over $\End_\sfU(FX)$. For the converse, observe that each compact object in $\sfT$ is isomorphic to a direct summand of an object of the form $F'C$ for some $C\in\sfU^c$.
\end{proof}

\begin{proposition}
\label{pr:endofinite} Let $\fp$ be a point in $\Proj H^*(G,k)$ and $M$
a $G$-module that is compact in $\gam_{\fp}(\StMod G)$.  Then $M$ is
endofinite both in $\StMod G$ and in $\gam_{\fp}(\StMod G)$.
\end{proposition}

\begin{proof} 
  By Corollary~\ref{co:realisability}, the module $M$ is of the form
  $N\da_G$ for a finite dimensional $G_K$-module $N$.  Clearly, $N$ is
  endofinite in $\StMod G_K$ and $(-)\da_G$ preserves products and
  coproducts, so it follows by Lemma~\ref{le:endofinite} that $M$ is
  endofinite in $\StMod G$. By the same token, as the inclusion
  $(\StMod G)_\fp\to \StMod G$ preserves products and coproducts, $M$
  is endofinite in $(\StMod G)_\fp$ as well. Finally, the functor
  $\gam_{\mcV(\fp)}$ is a right adjoint to the inclusion
  $\gam_\fp(\StMod G)\to (\StMod G)_\fp$. It preserves products, being
  a right adjoint, and also coproducts. Thus $M$ is endofinite in
  $\gam_\fp(\StMod G)$, again by Lemma~\ref{le:endofinite}.
\end{proof}

\section{$G$-modules and Tate duality}

Now we turn to various dualities for modules over finite group
schemes. We begin by recalling the construction of the transpose and
the dual of a module over a finite group scheme, and certain functors
associated with them. Our basic reference for this material is
Skowro\'nski and Yamagata~\cite[Chapter
III]{Skowronski/Yamagata:2011a}.

Throughout $G$ will be a finite group scheme over $k$. We write
$(-)^\vee=\Hom_k(-,k)$.

\subsection*{Transpose and dual} 

Let $G^{\op}$ be the opposite group scheme of $G$; it can be realised
as the group scheme associated to the cocommutative Hopf algebra
$(kG)^{\op}$. Since $kG$ is a $G$-bimodule, the assignment
$M\mapsto \Hom_{G}(M,kG)$ defines a functor
\[ (-)^{t}\col \Mod G \lra \Mod G^{\op}\,.
\] Let now $M$ be a finite dimensional $G$-module.  Give a minimal
projective presentation $P_{1}\xra{f}P_{0}\to M$, the \emph{transpose}
of $M$ is the $G^{\op}$-module $\Tr M:= \Coker(f^{t})$. By
construction, there is an exact sequence of $G^{\op}$-modules:
\[ 0\lra M^{t} \lra P_{0}^{t}\xra{\ f^{t}\ }P_{1}^{t}\lra \Tr M\lra
0\,.
\] The $P_{i}^{t}$ are projective $G^{\op}$-modules, so this yields an
isomorphism of $G^{\op}$-modules
\begin{equation*}
\label{eq:t=omega2} M^{t}\cong \Omega^{2}\Tr M\,.
\end{equation*}

Given a $G^{\op}$-module $N$, the $k$-vector space $\Hom_{k}(N,k)$ has
a natural structure of a $G$-module, and the assignment $N\mapsto
\Hom_{k}(N,k)$ yields a functor
\[ D:=\Hom_{k}(-,k)\col \stmod G^{\op}\lra \stmod G\,.
\]

\subsection*{The Auslander-Reiten translate} 

In what follows we write $\tau$ for the Auslander-Reiten translate of
$G$:
\[ \tau := D\circ \Tr \col \stmod G\to \stmod G
\]

Given an extension of fields $K/k$, for any finite dimensional
$G$-module $M$ there is a stable isomorphism of $G_{K}$-modules
\begin{equation*}
\label{eq:AR-extension} (\tau M)_{K}\cong \tau (M_{K})\,.
\end{equation*}

\subsection*{Nakayama functor}

The Nakayama functor
\[ \nu\col\Mod G\xra{\ \sim\ }\Mod G
\] is given by the assignment
\[ M\mapsto D(kG)\otimes_{kG} M\cong \delta_{G}\otimes_{k} M\,.
\] where $\delta_G=\nu(k)$ is the \emph{modular character} of $G$; see
\cite[I.8.8]{Jantzen:2003a}. Since the group of characters of $G$ is
finite, by \cite[\S2.1 \& \S2.2]{Waterhouse:1979a}, there exists a
positive integer $d$ such that $\delta_{G}^{\otimes d}\cong k$ and
hence as functors on $\Mod G$ there is an equality
\begin{equation}
\label{eq:nu-periodic} \nu^{d} = \id\,.
\end{equation} 
When $M$ is a finite dimensional $G$-module, there are natural stable
isomorphisms
\[ \nu M\cong D(M^t)\cong \Omega^{-2}\tau M\,.
\] When in addition $M$ is projective, one has
\begin{equation*}
\label{eq:Nak-duality}
\Hom_G(M,-)^\vee\cong(M^t\otimes_{kG}-)^\vee\cong\Hom_G(-,\nu M).
\end{equation*}

Let $K/k$ be an extension of fields.  For any $G$-module $M$ there is
a natural isomorphism of $G_{K}$-modules
\begin{equation}
\label{eq:Nak-extension} (\nu M)_{K}\cong \nu (M_{K})\,.
\end{equation} 
This is clear for $M=kG$ since
\[K\otimes_k \Hom_k(kG,k)\cong\Hom_k(kG,K)\cong \Hom_K(K\otimes_k
kG,K)\, ,\] and the general case follows by taking a free presentation
of $M$.

\begin{remark} 
  We have $\delta_G=k$ if and only if the algebra $kG$ is
  symmetric. In particular, $\delta_G=k$ when $G$ is a finite discrete
  group.
\end{remark}

\subsection*{Tate duality} 

For finite groups, the duality theorem below is classical and due to
Tate \cite[Chapter XII, Theorem 6.4]{Cartan/Eilenberg:1956a}.  An
argument for the extension to finite group schemes was sketched in
\cite[\S2]{Benson/Iyengar/Krause/Pevtsova:2016a}, and is reproduced
here for readers convenience.

\begin{theorem}
\label{thm:tate} Let $G$ be a finite group scheme over a field
$k$. For any $G$-modules $M,N$ with $M$ finite dimensional, there are
natural isomorphisms
\[ \sHom_{G}(M,N)^{\vee}\cong \sHom_{G}(N,\Omega^{-1}\tau M)\cong
\sHom_{G}(N,\Omega\nu M)\,.
\]
\end{theorem}

\begin{proof} 
  A formula of Auslander and Reiten~\cite[Proposition
  I.3.4]{Auslander:1978a}, see also \cite[Corollary
  p. 269]{Krause:2003a}, yields the first isomorphism below
\[ \sHom_{G}(M,N)^{\vee}\cong \Ext^{1}_{G}(N,\tau M) \cong
\sHom_{G}(N,\Omega^{-1}\tau M)
\] The second isomorphism is standard. It remains to recall that $\tau
M \cong \Omega^{2}\nu M$.
\end{proof}

Restricted to finite dimensional $G$-modules, Tate duality is the
statement that the $k$-linear category $\stmod G$ has Serre duality,
with Serre functor $\Omega\nu$. A refinement of this Serre duality
will be proved in Section~\ref{sec:serre}.

\section{Local cohomology versus injective cohomology}
\label{sec:Duality} 

Let $k$ be a field and $G$ a finite group scheme over $k$. In this
section we establish the main result of this work; it identifies for a
prime ideal $\fp$ in $H^{*}(G,k)$, up to some twist and some
suspension, the local cohomology object $\gam_\fp(k)$ with the
injective cohomology object $T_\fp(k)$.

\begin{theorem}
\label{thm:gorenstein} Fix a point $\fp$ in $\Proj H^{*}(G,k)$ and let
$d$ be the Krull dimension of $H^{*}(G,k)/\fp$. There is a stable
isomorphism of $G$-modules
\[ \gam_{\fp}(\delta_{G}) \cong \Omega^{-d}T_{\fp}(k)\,;
\] equivalently, for any $G$-module $M$ there is a natural isomorphism
\[ \sHom_G(M,\Omega^d \gam_{\fp}(\delta_{G})
)\cong\Hom_{H^*(G,k)}(H^*(G,M),I(\fp))\,.
\]
\end{theorem}

When $G$ is the group scheme arising from a finite group the result
above was proved by Benson and
Greenlees~\cite[Theorem~2.4]{Benson/Greenlees:2008a} using Gorenstein
duality for cochains on $BG$, the classifying space of $G$.
Benson~\cite[Theorem 2]{Benson:2008a} gave a different proof by
embedding $G$ into a general linear group and exploiting the fact that
its cohomology ring is a polynomial ring, as was proved by
Quillen. There is an extension of these results to compact Lie groups;
see \cite[Theorem~6.10]{Benson/Greenlees:2014a}, and recent work of
Barthel, Heard, and
Valenzuela~\cite[Proposition~4.33]{Barthel/Heard/Valenzuela:2016a}.

Theorem~\ref{thm:gorenstein} is established using (by necessity)
completely different arguments, thereby giving yet another proof in
the case of finite groups that is, in a sense, more elementary than
the other ones for it is based on classical Tate duality.

A caveat: In \cite{Benson:2008a, Benson/Greenlees:2008a} it is
asserted that $\gam_{\fp}(k)\cong \Omega^{d}T_{\fp}(k)$. However, this
is incorrect and the correct shift is the one in the preceding
theorem. We illustrate this by computing these modules directly for
the quaternions.

\begin{example}
\label{exa:Q8} Let $G:=Q_{8}$, the quaternions, viewed as a group
scheme over a field $k$ of characteristic $2$. In this case
$\delta_{G}=k$, the trivial character. The cohomology algebra of $G$
is given by
\[ H^{*}(G,k) = k[z]\otimes_{k}B \quad \text{where
$B=k[x,y]/(x^{2}+xy+y^{2},x^{2}y+xy^{2})$}\,,
\] with $|x|=1=|y|$ and $|z|=4$; see, for instance,
\cite[p.~186]{Benson:1984a}. Thus $\Proj H^{*}(G,k)$ consists of a
single point, namely $\fm:=(x,y)$.  In particular, $\gam_{\fm}k = k$,
in $\StMod G$.

Next we compute $I(\fm)$ as a module over $H^{*}(G,k)_{\fm} \cong
k[z^{\pm1}] \otimes_{k} B$, using Lemma~\ref{lem:hull2}. The extension
$k[z^{\pm 1}]\subseteq k[z^{\pm1}]\otimes_{k}B$ is evidently finite
(and hence also residually finite). Since $\fm\cap k[z^{\pm1}]=(0)$
and $k[z^{\pm 1}]$ is a graded field, from Lemma~\ref{lem:hull2} one
gets an isomorphism of $H^{*}(G,k)_{\fm}$-modules.
\begin{align*} I(\fm) &\cong
\Hom_{k[z^{\pm1}]}(k[z^{\pm1}]\otimes_{k}B,k[z^{\pm1}]) \\ &\cong
k[z^{\pm1}]\otimes_{k}\Hom_{k}(B,k) \\ &\cong k[z^{\pm 1}]\otimes_{k}
\Sigma^{3}B \\ &\cong \Sigma^{3} H^{*}(G,k)_{\fm}
\end{align*} This yields the first isomorphism below of $G$-modules
\[ T_{\fm}(k) \cong \Omega^{-3} k \cong \Omega^{1} k\, ,
\] and the second one holds because $\Omega^{4}k\cong k$ in $\StMod
G$.
\end{example}

\begin{proof}[Proof of Theorem~\ref{thm:gorenstein}] It follows from
\eqref{eq:Tate-local} that for any $\fp$-local $H^{*}(G,k)$-module
$I$, there is an isomorphism
\[ \Hom_{H^*(G,k)}(H^*(G,M),I)\cong \Hom_{H^*(G,k)}(\sHom_G^*(k,M),I).
\] Consequently, one can rephrase the defining isomorphism
\eqref{eq:inj-coh} for the object $T_\fp(k)$ as a natural isomorphism
\[ \sHom_G(M,T_\fp(k) )\cong\Hom_{H^*(G,k)}(H^*(G,M),I(\fp))\,.
\] Therefore, the main task is to prove that there is a stable
isomorphism:
\[ \gam_{\fp}(\nu k)\cong \Omega^{-d}T_{\fp}(k)\,.
\] Recall that $\nu k =\delta_{G}$.

We first verify the isomorphism above for closed points of $\Proj
H^{*}(G,k)$ and then use a reduction to closed points. The proof uses
the following simple observation: For any $G$-modules $X$ and $Y$ that
are $\fp$-local and $\fp$-torsion, there is an isomorphism $X\cong Y$
in $\StMod G$ if and only if there is a natural isomorphism
\[ \sHom_G(M,X)\cong \sHom_G(M,Y)
\] for $\fp$-local and $\fp$-torsion $G$-modules $M$. This follows
from Yoneda's lemma.

\begin{claim} 
  The desired isomorphism holds when $\fm$ is a closed point.
\end{claim}

Set $A:=H^{*}(G,k)$ and $R:=A_{\fm}$. The injective hull, $I(\fm)$, of
the $A$-module $A/\fm$ is the same as that of the $R$-module $k(\fm)$,
viewed as an $A$-module via restriction of scalars along the
localisation map $A\to R$. Thus $I(\fm)$ is the module $I$ described
in Lemma~\ref{lem:hull}. Let $M$ be a $G$-module that is $\fm$-local
and $\fm$-torsion. The claim is a consequence of the following
computation:
\begin{align*} \sHom_{G}(M, \Omega\gam_{\fm}(\nu k)) & \cong
\sHom_{G}(M, \Omega\nu k) \\ & \cong \sHom_{G}(k,M)^{\vee} \\ & \cong
\Hom_{R}(\sHom_{G}^{*}(k,M),I(\fm)) \\ & \cong
\Hom_{A}(\sHom_{G}^{*}(k,M),I(\fm)) \\ & \cong \sHom_G(M,T_\fm(k))
\end{align*} 
The first isomorphism holds because $M$ is $\fm$-torsion; the second
is Tate duality, Theorem~\ref{thm:tate}, and the next one is by
Lemma~\ref{lem:hull}, which applies because $\sHom_{G}^{*}(k,M)$ is
$\fm$-local and $\fm$-torsion as an $A$-module.

\medskip

Let $\fp$ be a point in $\Proj H^{*}(G,k)$ that is not closed, and let
$K$, $\bsb$, and $\fm$ be as in Construction~\ref{con:generic}.
Recall that $\fm$ is a closed point in $H^{*}(G_{K},K)$ lying over
$\fp$.

\begin{claim} 
  There is a stable isomorphism of $G$-modules
\begin{equation}
\label{eq:claim-T-module} (\kos {T_{\fm}(K)}{\bsb})\da_{G} \cong
\Omega^{-d+1} T_{\fp}(k)
\end{equation} 
where $d$ is the Krull dimension of $H^{*}(G,k)/\fp$.
\end{claim}

Let $M$ be a $G$-module that is $\fp$-local and $\fp$-torsion. Then
Theorem~\ref{thm:realisability} applies and yields isomorphisms of
$G$-modules.
\[ (\kos{M_{K}}{\bsb})\da_{G} \cong (M_{K}\otimes_{K} \kos
K{\bsb})\da_{G} \cong M\,.
\] This gives the sixth isomorphism below.
\begin{align*} 
\sHom_{G}(M, \Omega^{d-1}(\kos
{T_{\fm}(K)}{\bsb})\da_{G}) &\cong
\sHom_{G_{K}}(M_{K},\Omega^{d-1}(\kos {T_{\fm}(K)}{\bsb})) \\ &\cong
\sHom_{G_{K}}(\kos{M_{K}}{\bsb},T_{\fm}(K)) \\ &\cong
\Hom_{H^{*}(G_{K},K)}(\sHom_{G_{K}}^{*}(K, \kos{M_{K}}{\bsb}),I(\fm))
\\ &\cong \Hom_{H^{*}(G,k)} (\sHom_{G_{K}}^{*}(K,
\kos{M_{K}}{\bsb}),I(\fp)) \\ &\cong
\Hom_{H^{*}(G,k)}(\sHom_{G}^{*}(k, (\kos{M_{K}}{\bsb})\da_{G}),I(\fp))
\\ &\cong \Hom_{H^{*}(G,k)}(\sHom_{G}^{*}(k,M),I(\fp)) \\ &\cong
\sHom_G(M, T_\fp(k))
\end{align*} 
The first and the fifth isomorphisms are by adjunction. The second
isomorphism is a direct computation using \eqref{eq:kos-swap} and
\eqref{eq:periodicity}. The next one is by definition and the fourth
isomorphism is by Lemma~\ref{lem:hull2} applied to the canonical
homomorphism $H^{*}(G,k)\to H^{*}(G_{K},K)$; note that the
$H^{*}(G_{K},K)$-module $\sHom_{G_{K}}^{*}(K, \kos{M_{K}}{\bsb})$ is
$\fm$-torsion. This justifies the claim.

\medskip

We are now ready to wrap up the proof of the theorem. Since $\fm$ is a
closed point in $\Proj H^{*}(G_{K},K)$, the first claim yields that
the $G_{K}$-modules $\gam_{\fm}(\nu K)$ and $\Omega^{-1} T_{\fm}(K)$
are isomorphic. This then gives an isomorphism of $G_{K}$-modules.
\[ \nu K \otimes_{K} \gam_{\fm}(\kos K{\bsb}) \cong \kos
{\gam_{\fm}(\nu K)}{\bsb} \cong \Omega^{-1} \kos {T_{\fm}(K)}{\bsb}
\] Restricting to $G$ and applying \eqref{eq:claim-T-module} gives the
last of the following isomorphisms of $G$-modules.
\begin{align*} 
\gam_{\fp}(\nu k) &\cong ( (\nu k)_K \otimes_{K}
\gam_{\fm}(\kos K{\bsb} )\da_{G}\\ &\cong (\nu K \otimes_{K}
\gam_{\fm}(\kos K{\bsb}))\da_{G} \\ &\cong \Omega^{-d}T_{\fp}(k)
\end{align*} 
The first one is by Theorem~\ref{thm:realisability} and the second by
\eqref{eq:Nak-extension}.
\end{proof}

The following consequences of Theorem~\ref{thm:gorenstein} were
anticipated in \cite[pp.~203]{Benson:2001a}.  For a graded module
$M=\bigoplus_{n\in\bbZ} M^n$ and $i\in\bbZ$ the twist $M(i)$ is the
graded module with $M(i)^n=M^{n+i}$.

\begin{corollary} 
  Fix $\fp$ in $\Proj H^{*}(G,k)$. With $d$ the Krull dimension of
  $H^{*}(G,k)/\fp$ there are isomorphisms of $H^{*}(G,k)$-modules
\[ H^{*}(G,\gam_{\fp}(k)) \cong
\Hom^{*}_{H^{*}(G,k)}(H^{*}(G,\delta_{G}),I(\fp))(d)
\] and
\[ H^{*}(G,\End_{k}(\gam_{\fp}(k))) \cong (H^{*}(G,k)_{\fp})^{\wedge}
\] where $(-)^{\wedge}$ denotes completion with respect to the
$\fp$-adic topology.
\end{corollary}

\begin{proof} 
  Remark~\ref{re:Tate-coh} will be used repeatedly and without
  comment. Set $R=H^*(G,k)$. Since $\gam_{\fp}(k)$ is $\fp$-local, the
  $R$-modules $H^{*}(G,\gam_{\fp}(k))$ and
  $\sHom_{G}^*(k,\gam_{\fp}(k))$ are isomorphic. The first of the
  stated isomorphisms is a composition of the following isomorphisms
  of $R$-modules.
\begin{align*} 
\sHom^{*}_{G}(k,\gam_{\fp}(k)) &\cong
\sHom^{*}_{G}(\delta_G,\gam_{\fp}(\delta_G))\\ &\cong
\sHom^{*}_{G}(\delta_{G},\Omega^{-d} T_\fp(k)) \\ &\cong
\sHom^{*}_{G}(\delta_{G},T_\fp(k))(d) \\ &\cong
\sHom^{*}_{R}(H^*(G,\delta_{G}),I(\fp))(d)
\end{align*} 
The second isomorphism is by Theorem~\ref{thm:gorenstein}, the one
after is standard, while the last one is by the definition of
$T_{\fp}(k)$.

In the same vein, one has the following chain of isomorphisms.
\begin{align*} 
\sHom^{*}_{G}(k,\End_{k}(\gam_{\fp}(k))) & \cong
\sHom^{*}_{G}(\gam_{\fp}(k),\gam_{\fp}(k)) \\ & \cong
\sHom^{*}_{G}(T_\fp(k),T_\fp(k)) \\ &\cong
\Hom^{*}_{R}(\sHom^{*}_{G}(k,T_\fp(k)),I(\fp))\\ &\cong
\Hom^{*}_{R}(I(\fp),I(\fp)) \\ &\cong (R_{\fp})^{{\wedge}}
\end{align*} 
The second isomorphism is by Theorem~\ref{thm:gorenstein} and the rest
are standard.
\end{proof}

\begin{remark} 
  Another consequence of Theorem~\ref{thm:gorenstein} is that
  $\gam_{\fp}(\delta_{G})$, and hence also $\gam_{\fp}k$, is an
  indecomposable pure injective object in $\StMod G$; see
  \cite[Theorem~5.1]{Benson/Krause:2002a}.
\end{remark}

\section{The Gorenstein property}
\label{sec:gorenstein}

In this section we introduce a notion of a Gorenstein triangulated
category and reinterpret Theorem~\ref{thm:gorenstein} to mean that
$\StMod G$ has this property. The rationale for doing so will become
clear in the next section.
 
\subsection*{The Gorenstein property} 

Let $R$ be a graded commutative noetherian ring and $\sfT$ a compactly
generated $R$-linear triangulated category.  The \emph{support} of
$\sfT$ is by  definition the set
\[\supp_R(\sfT)=\{\fp\in\Spec R\mid \gam_\fp\neq 0\}\,.\]
We say that $\sfT$ is
\emph{Gorenstein} if there is an $R$-linear triangle equivalence
\[ F\col \sfT^{\sfc}\xra{\sim}\sfT^{\sfc}
\] and for every $\fp$ in $\supp_R(\sfT)$ there is an integer $d(\fp)$ and a
natural isomorphism
\[ \gam_\fp\comp F\cong \Si^{d(\fp)} \comp T_\fp
\] of functors $\sfT^{\sfc}\to \sfT$. In this context we call $F$ a
\emph{global Serre functor}, because in
Proposition~\ref{pr:gorenstein-serre} we show that localising at $\fp$
induces a Serre functor $\Si^{-d(\fp)}F_\fp$ in the sense of Bondal
and Kapranov \cite{Bondal/Kapranov:1990a}.

Let now $\sfT=(\sfT,\otimes,\one)$ be a tensor triangulated category
such that $R$ acts on $\sfT$ via a homomorphism of graded rings
$R\to \End_{\sfT}^{*}(\one)$. In that case the Gorenstein property is
captured by the existence of a distinguished object.

\begin{lemma}
\label{le:gor-tt} The $R$-linear triangulated category $\sfT$ is
Gorenstein if and only if there exists a compact object $W$ with the
following properties:
\begin{enumerate}[\quad\rm(1)]
\item There is a compact object $W^{-1}$ such that $W\otimes
W^{-1}\cong \one$;
\item For each $\fp$ in $\supp_R(\sfT)$ there exists an integer
  $d(\fp)$ and an isomorphism
\[ \gam_{\fp}W \cong \Si^{d(\fp)} T_{\fp}(\one)\,.
\]
\end{enumerate}
\end{lemma}

One may think of $W$ as a \emph{dualising} object in $\sfT$.

\begin{proof} Since $\gam_\fp$ and $T_\fp$ are tensor functors, by
\eqref{eq:tensor-identity}, the functor $F\colon\sfT^c\to\sfT^c$ and
the object $W$ determine each other: $W:= F(\one)$ and $F:=W\otimes
-$.
\end{proof}

The example below justifies the language of Gorenstein triangulated
categories.

\begin{example}
\label{ex:gor-ring} Let $A$ be a commutative noetherian ring and
$\sfD$ the derived category of $A$-modules. This is an $A$-linear
compactly generated tensor triangulated category, with compact objects
the perfect complexes of $A$-modules, that is to say, complexes
quasi-isomorphic to bounded complexes of finitely generated projective
$A$-modules.

Recall that the ring $A$ is \emph{Gorenstein} if for each $\fp\in\Spec
A$ the injective dimension of $A_{\fp}$, as a module over itself, is
finite; see \cite[3.1]{Bruns/Herzog:1998a}. By Grothendieck's local
duality theorem~\cite[\S3.5]{Bruns/Herzog:1998a}, this is equivalent
to an isomorphism of $A_{\fp}$-modules
\[ \gam_{\fp} A \cong \Si^{-\dim A_{\fp}} I(\fp) \,.
\] Thus $\sfD$ is Gorenstein with dualising object $A$ and
$d(\fp)=-\dim A_{\fp}$; see Lemma~\ref{le:gor-tt}. Conversely, it is
not difficult to check that $\sfD$ is Gorenstein only if $A$ is
Gorenstein.
\end{example}

There are other examples (notably, differential graded algebras) that
fit into this context but we defer discussing these to another
occasion.  For a finite group scheme $G$ over a field $k$, the
Gorenstein property for $\StMod G$ is basically a reformulation of
Theorem~\ref{thm:gorenstein}.

\begin{corollary}
\label{cor:stmod-serre} As an $H^{*}(G,k)$-linear triangulated
category, $\StMod G$ is Gorenstein, with $F=\delta_G\otimes_k -$ the
Nakayama functor and $d(\fp)=\dim H^*(G,k)/\fp$ for each $\fp$ in
$\Proj H^*(G,k)$.\qed
\end{corollary}

Next we discuss the Gorenstein property for $\KInj G$ for a finite
group scheme $G$.  To this end observe that the assignment $X\mapsto
\nu X$ induces triangle equivalences
\[ \KInj G\xra{ \sim }\KInj G\qquad\text{and}\qquad \sfD^b(\mod
G)\xra{ \sim }\sfD^b(\mod G).
\] We are ready to establish the Gorenstein property for $\KInj G$.

\begin{theorem}
\label{th:KInj} Let $G$ be a finite group scheme over a field
$k$. Then $\KInj G$ is Gorenstein with $F$ induced by the Nakayama
functor and $d(\fp)=\dim H^*(G,k)/\fp$ for each $\fp\in\Spec
H^*(G,k)$.
\end{theorem}

\begin{proof} 
  Set $\fm:=H^{\ges 1}(G,k)$. By Lemma~\ref{le:KInj-p-local}, one has
  $\gam_\fm= \sfp k\otimes_k-$, where $\sfp k$ denotes a projective
  resolution of the trivial representation. The isomorphism
  $\gam_\fm\comp F\cong T_\fm$ follows
  from~\cite[Theorem~3.4]{Krause/Le:2006a}.

For a prime ideal $\fp\neq\fm$, the assertion follows from
Theorem~\ref{thm:gorenstein}, for localisation at $\fp$ yields a
triangle equivalence $\gam_\fp(\KInj G)\xra{\sim}\gam_\fp(\StMod G)$
by Lemma~\ref{le:KInj-p-local}.
\end{proof}

\section{Local Serre duality}
\label{sec:serre} 

In this section we introduce a notion of local Serre duality for an
essentially small $R$-linear triangulated category and link it to the
Gorenstein property from Section~\ref{sec:gorenstein}.  We use the
concept of a Serre functor for a triangulated category which is due to
Bondal and Kapranov \cite{Bondal/Kapranov:1990a}; this provides a
conceptual way to formulate classical Serre duality and Grothendieck's
local duality in a triangulated setting.

In the second part of this section we discuss the existence of
Auslander-Reiten triangles. These were introduced by Happel for
derived categories of finite dimensional algebras \cite{Happel:1988a},
and in \cite{Happel:1991a} he established the connection with the
Gorenstein property, while Reiten and Van den Bergh
\cite{Reiten/VandenBergh:2002a} discovered the connection between
Auslander-Reiten triangles and the existence of a Serre functor.

\subsection*{Small triangulated categories with central action} 

Let $\sfC$ be an essentially small $R$-linear triangulated
category. Fix $\fp\in \Spec R$ and let $\sfC_\fp$ denote the
triangulated category that is obtained from $\sfC$ by keeping the
objects of $\sfC$ and setting
\[ \Hom_{\sfC_\fp}^*(X,Y):=\Hom^*_{C}(X,Y)_\fp\,.
\] Then $\sfC_{\fp}$ is an $R_{\fp}$-linear triangulated category and
localising the morphisms induces an exact functor $\sfC\to\sfC_\fp$.

Let $\tors_{\fp}\sfC$ be the full subcategory of $\fp$-torsion objects
in $\sfC_\fp$, namely
\[ \tors_{\fp}\sfC:=\{X\in \sfC_\fp\mid \End^{*}_{\sfC_{\fp}}(X)
\text{ is $\fp$-torsion} \}.
\] This is a thick subcategory of $\sfC_\fp$. In
\cite{Benson/Iyengar/Krause:2015a} this category is denoted
$\gam_{\fp}\sfC$. The notation has been changed to avoid confusion.

\begin{remark} 
  Let $F\colon \sfC\to\sfC$ be an $R$-linear equivalence. It is
  straightforward to check that this induces triangle equivalences
  $F_\fp\colon\sfC_\fp\xra{\sim}\sfC_\fp$ and
  $\tors_{\fp}\sfC\xra{\sim}\tors_{\fp}\sfC$ making the following
  diagram commutative.
\[
\begin{tikzcd} 
\sfC\arrow[r,two heads]\arrow[d,"F"]&\sfC_\fp
\arrow[d,"F_\fp"]&\arrow[l, tail]\tors_{\fp}\sfC \arrow[d,"F_\fp"]\\
\sfC\arrow[r,two heads]&\sfC_\fp&\arrow[l,tail]\tors_{\fp}\sfC
\end{tikzcd}
\]
\end{remark}

\begin{remark}
 \label{rem:compacts-localisation} Let $\sfT$ be a compactly generated
$R$-linear triangulated category. Set $\sfC:=\sfT^c$ and fix
$\fp\in\Spec R$. The triangulated categories $\sfT_\fp$ and
$\gam_\fp\sfT$ are compactly generated. The left adjoint of the
inclusion $\sfT_\fp\to\sfT$ induces (up to direct summands) a triangle
equivalence $\sfC_\fp\xra{\sim}(\sfT_\fp)^c$ and restricts to a
triangle equivalence
\[ \tors_{\fp}\sfC\xra{\sim}(\gam_\fp\sfT)^c\,.
\] This follows from the fact that the localisation functor
$\sfT\to\sfT_\fp$ taking $X$ to $X_\fp$ preserves compactness and that
for compact objects $X,Y$ in $\sfT$
\[ \Hom_\sfT^*(X,Y)_\fp\xra{\sim}\Hom_{\sfT_\fp}^*(X_\fp,Y_\fp)\,.
\] For details we refer to \cite{Benson/Iyengar/Krause:2015a}.
\end{remark}

\subsection*{Local Serre duality}

Let $R$ be a graded commutative ring that is \emph{local}; thus there
is a unique homogeneous maximal ideal, say $\fm$.  Extrapolating from
Bondal and Kapranov \cite[\S3]{Bondal/Kapranov:1990a}, we call an
$R$-linear triangle equivalence $F\colon\sfC\xra{\sim}\sfC$ a
\emph{Serre functor} if for all objects $X,Y$ in $\sfC$ there is a
natural isomorphism
\begin{equation}
\label{eq:serre}
\Hom_{R}(\Hom^*_{\sfC}(X,Y),I(\fm))\xra{\sim}\Hom_{\sfC}(Y,F X)\,.
\end{equation} 
The situation when $R$ is a field was the one considered in
\cite{Bondal/Kapranov:1990a}. For a general ring $R$, the appearance
of $\Hom^{*}_{R}(-,I(\fm))$, which is the Matlis duality functor, is
natural for it is an extension of vector-space duality; see also
Lemma~\ref{lem:hull}.

For an arbitrary graded commutative ring $R$, we say that an
$R$-linear triangulated category $\sfC$ satisfies \emph{local Serre
duality} if there exists an $R$-linear triangle equivalence
$F\col\sfC\xra{\sim}\sfC$ such that for every $\fp\in\Spec R$ and some
integer $d(\fp)$ the induced functor $\Si^{-d(\fp)}
F_\fp\colon\tors_{\fp}\sfC\xra{\sim}\tors_{\fp}\sfC$ is a Serre
functor for the $R_\fp$-linear category $\tors_{\fp}\sfC$. Thus for
all objects $X,Y$ in $\tors_{\fp}\sfC$ there is a natural isomorphism
\[
\Hom_{R}(\Hom^*_{\sfC_{\fp}}(X,Y),I(\fp))\xra{\sim}\Hom_{\sfC_{\fp}}(Y,\Si^{-d(\fp)}
F_\fp X)\,.
\]

For a compactly generated triangulated category, the Gorenstein
condition from Section~\ref{sec:gorenstein} is linked to local Serre
duality for the subcategory of compact objects.

\begin{proposition}
  \label{pr:gorenstein-serre} Let $R$ be a graded commutative
  noetherian ring and $\sfT$ a compactly generated $R$-linear
  triangulated category. Suppose that $\sfT$ is Gorenstein, with
  global Serre functor $F$ and shifts $\{d(\fp)\}$. Then for each
  $\fp\in\supp_R(\sfT)$, object $X\in (\gam_{\fp}\sfT)^c$ and
  $Y\in \sfT_{\fp}$ there is a natural isomorphism
\[ \Hom_R(\Hom^*_{\sfT}(X,Y),I(\fp)) \cong
\Hom_{\sfT}(Y,\Si^{-d(\fp)}F_{\fp}(X))\,.
\]
\end{proposition}

\begin{proof} 
  Given Remark~\ref{rem:compacts-localisation} we can assume
  $X=C_{\fp}$ for a $\fp$-torsion compact object $C$ in $\sfT$.  The
  desired isomorphism is a concatenation of the following natural
  ones:
\begin{align*} 
\Hom_R(\Hom^*_{\sfT}(C_{\fp},Y),I(\fp)) & \cong
\Hom_R(\Hom^*_{\sfT}(C,Y),I(\fp)) \\ {}&\cong\Hom_{\sfT}(Y,T_\fp(C))\\
{}&\cong\Hom_{\sfT}(Y,\Si^{-d(\fp)}\gam_\fp F(C))\\
{}&\cong\Hom_{\sfT}(Y,\Si^{-d(\fp)}\gam_{\mcV(\fp)} F(C)_\fp)\\
{}&\cong\Hom_{\sfT}(Y,\Si^{-d(\fp)} \gam_{\mcV(\fp)} F_\fp(C_\fp))\\
{}&\cong\Hom_{\sfT}(Y,\Si^{-d(\fp)} F_\fp(C_{\fp}))
\end{align*} 
In this chain, the first map is induced by the localisation
$C\mapsto C_{\fp}$ and is an isomorphism because $Y$ is $\fp$-local.
The second one is by the definition of $T_{\fp}(C)$; the third is by
the Gorenstein property of $\sfT$; the fourth is by the definition of
$\gam_{\fp}$; the last two are explained by
Remark~\ref{rem:compacts-localisation}, where for the last one uses
also the fact that $C_{\fp}$, and hence also $F_{\fp}(C_{\fp})$, is
$\fp$-torsion.
\end{proof}
 
\begin{corollary}
\label{co:gorenstein-serre} Let $R$ be a graded commutative noetherian
ring and $\sfT$ a compactly generated $R$-linear triangulated
category. If $\sfT$ is Gorenstein, then $\sfT^{\sfc}$ satisfies local
Serre duality.
\end{corollary}
\begin{proof} Combine Proposition~\ref{pr:gorenstein-serre} with
Remark~\ref{rem:compacts-localisation}.
\end{proof}

\begin{example}
\label{ex:gor-ring2} In the notation of Example~\ref{ex:gor-ring},
when $A$ is a (commutative noetherian) Gorenstein ring, local Serre
duality reads: For each $\fp\in \Spec A$ and integer $n$ there are
natural isomorphisms
\[ \Hom_{A_{\fp}}(\Ext^{n}_{A_{\fp}}(X,Y),I(\fp)) \cong \Ext^{n+\dim
A_{\fp}}_{A_{\fp}}(Y, X)
\] where $X$ is a perfect complexes of $A_{\fp}$-modules with finite
length cohomology, and $Y$ is a complex of $A_{\fp}$-modules.
\end{example}

\subsection*{Auslander-Reiten triangles} 

Let $\sfC$ be an essentially small triangulated category. Following
Happel \cite{Happel:1988a}, an exact triangle
$X\xra{\alpha} Y\xra{\beta} Z\xra{\gamma}\Si X$ in $\sfC$ is an
\emph{Auslander-Reiten triangle} if
\begin{enumerate}[\quad\rm(1)]
\item any morphism $X\to X'$ that is not a split monomorphism factors
through $\alpha$;
\item any morphism $Z'\to Z$ that is not a split epimorphism factors
through $\beta$;
\item $\gamma\neq 0$.
\end{enumerate} 
In this case, the endomorphism rings of $X$ and $Z$ are local, and in
particular the objects are indecomposable. Moreover, each of $X$ and
$Z$ determines the AR-triangle up to isomorphism. Assuming conditions
(2) and (3), the condition (1) is equivalent to the following:
\begin{enumerate}[\quad\rm(1)]
\item[\rm(1$'$)] The endomorphism ring of $X$ is local.
\end{enumerate} 
See \cite[\S2]{Krause:2000a} for details.

Let $\sfC$ be a \emph{Krull-Schmidt category}, that is, each object
decomposes into a finite direct sum of objects with local endomorphism
rings. We say that $\sfC$ \emph{has AR-triangles} if for every
indecomposable object $X$ there are AR-triangles
\[ V\to W\to X\to\Si V \quad\text{and}\quad X\to Y\to Z\to\Si X\,.
\]

The next proposition establishes the existence of AR-triangles; it is
the analogue of a result of Reiten and Van den Bergh
\cite[I.2]{Reiten/VandenBergh:2002a} for triangulated categories that
are Hom-finite over a field.

\begin{proposition}
  \label{pro:AR-existence} Let $R$ be a graded commutative ring that
is local, and let $\sfC$ be an essentially small $R$-linear
triangulated category that is Krull-Schmidt. If $\sfC$ has a Serre
functor, then it has AR-triangles.
\end{proposition}

\begin{proof} 
  Let $F$ be a Serre functor for $\sfC$ and $X$ an indecomposable
  object in $\sfC$. The ring $\End_{\sfC}(X)$ is thus local; let $J$
  be its maximal ideal and $I$ the right ideal of $\End^{*}_{\sfC}(X)$
  that it generates. Choose a nonzero morphism
  $\Hom^*_\sfC(X,X)/I\to I(\fm)$ and let $\gamma\colon X\to FX$ be the
  corresponding morphism in $\sfC$ provided by Serre
  duality~\eqref{eq:serre}. We claim that the induced exact triangle
\[ \Si^{-1}FX \to W\to X\xra{\gamma}FX
\] is an AR-triangle. Indeed, by construction, if $X'\to X$ is not a
split epimorphism, the composition $X'\to X\xra{\gamma}FX$ is zero.
Moreover $X$ is indecomposable so is $FX$.

Applying this construction to $F^{-1}\Si X$ yields an AR-triangle
starting at $X$.
\end{proof}

Let $R$ be a graded commutative ring.  An $R$-linear triangulated
category $\sfC$ is \emph{noetherian} if the $R$-module
$\Hom_\sfC^*(X,Y)$ is noetherian for all $X,Y$ in $\sfC$. This
property implies that for an object $X$ in $\tors_{\fp}\sfC$ the
$R_\fp$-module $E^{*}:=\End^*_{\tors_{\fp}\sfC}(X)$ is of finite
length. Therefore the graded ring $E^{*}$ is artinian, and so $E^{0}$
is an Artin algebra over $R_\fp^0$; see
\cite[Theorem~1.5.5]{Bruns/Herzog:1998a}.  In particular, the
idempotent completion of $\tors_{\fp}\sfC$ is a Krull-Schmidt
category.

\begin{corollary}
\label{co:noetherian} Let $\sfT$ be a compactly generated $R$-linear
triangulated category with $\sfT^c$ noetherian. If $\sfT$ is
Gorenstein, then $(\gam_\fp\sfT)^c$ has AR-triangles for $\fp$ in
$\Spec R$.
\end{corollary}

\begin{proof} 
  As explained above, the assumption that $\sfT^c$ is noetherian
  implies that $(\gam_\fp\sfT)^c$ is a Krull-Schmidt category. Thus
  the assertion follows by combining
  Propositions~\ref{pr:gorenstein-serre} and \ref{pro:AR-existence}.
\end{proof}
 
Next we consider local Serre duality for $\sfD^b(\mod G)$ and $\stmod
G$. Recall from Lemma~\ref{le:KInj-p-local} that localisation at
$\fp\in\Proj H^*(G,k)$ induces a triangle equivalence $\gam_\fp(\KInj
G)\xra{\sim}\gam_\fp(\StMod G)$. Using
Remark~\ref{rem:compacts-localisation}, this yields (up to direct
summands) triangle equivalences
\[ \tors_{\fp}(\sfD^b(\mod G))\xra{\sim}\gam_{\fp} (\KInj
G)^c\xra{\sim}\gam_{\fp} (\StMod G)^c \xra{\sim}\tors_{\fp}(\stmod G).
\]

The result below contains Theorem~\ref{thm:main1} in the introduction.

\begin{theorem}
\label{thm:stmodAR} Let $G$ be a finite group scheme over a field
$k$. Then the $H^*(G,k)$-linear triangulated category
$\sfD:=\sfD^b(\mod G)$ satisfies local Serre duality. Said otherwise,
given $\fp$ in $\Spec R$ and with $d$ the Krull dimension of
$H^*(G,k)/\fp$, for each $M$ in $\tors_{\fp}\sfD$ and $N$ in
$\sfD_{\fp}$, there are natural isomorphisms
\[\Hom_{H^*(G,k)}(\Hom^{*}_{\sfD_{\fp}}(M,N),I(\fp))
  \cong\Hom_{\sfD_{\fp}}(N,\Omega^{d}\delta_G\otimes_k M)\,.\] In
particular, $\tors_{\fp}\sfD$ has AR-triangles.
\end{theorem}

\begin{proof} 
  The first assertion follows from Theorem~\ref{th:KInj} and
  Proposition~\ref{pr:gorenstein-serre}. The existence of AR-triangles
  then follows from Corollary~\ref{co:noetherian}, as $\sfD$ is
  noetherian.
\end{proof}

\subsection*{AR-components and periodicity} 

The existence of AR-triangles for a triangulated category $\sfC$ gives
rise to an \emph{AR-quiver}; see for example
\cite{Happel/Preiser/Ringel:1982a, Liu:2010a}. The vertices are given
by the isomorphism classes of indecomposable objects in $\sfC$ and an
arrow $[X]\to [Y]$ exists if there is an irreducible morphism
$X\to Y$.

In the context of $\stmod G$, one can describe part of the structure
of the AR-quiver of the $\fp$-local $\fp$-torsion objects as the Serre
functor is periodic.

\begin{proposition} 
  Let $G$ be a finite group scheme over a field $k$. Fix a point $\fp$
  in $\Proj H^*(G,k)$ and set $d=\dim H^*(G,k)/\fp$. Then the Serre
  functor
  \[\Omega^{d}\nu \colon \tors_{\fp}(\stmod
    G)\xra{\sim}\tors_{\fp}(\stmod G)\] is periodic, that is,
  $(\Omega^{d}\nu)^r=\id$ for some positive integer $r$.
\end{proposition}

\begin{proof} 
  Lemma~\ref{le:periodicity} and \eqref{eq:nu-periodic} provide an
  integer $r\ge 0$ such that $\nu^rM\cong M$ and $\Omega^r M\cong M$
  for $M$ in $\tors_{\fp}(\stmod G)$. Thus $(\Omega^{d}\nu)^r=\id$,
  since $\nu$ and $\Omega$ commute.
\end{proof}

This has the following consequence.

\begin{corollary} 
  Every connected component of the AR-quiver of
  $\tors_{\fp}(\stmod G)$ is a stable tube in case it is infinite; and
  otherwise, it is of the form $\bbZ \Delta/ U$, where $\Delta$ is a
  quiver of Dynkin type and $U$ is a group of automorphisms of
  $\bbZ \Delta$.
\end{corollary}

\begin{proof} 
  Since the Serre functor on $\tors_{\fp}(\stmod G)$ is periodic, the
  desired result follows from \cite[Theorem~5.5]{Liu:2010a}; see also
  \cite{Happel/Preiser/Ringel:1982a}.
\end{proof}

The preceding result may be seen as a first step in the direction of
extending the results of Farnsteiner's~\cite[\S3]{Farnsteiner:2007a}
concerning $\stmod G$ to $\tors_{\fp}(\stmod G)$ for a general
(meaning, not necessarily closed) point $\fp$ of $\Proj H^{*}(G,k)$.

\appendix

\section{Injective modules at closed points}

In this section we collect some remarks concerning injective hulls
over graded rings, for use in Section~\ref{sec:Duality}. Throughout
$k$ will be a field and $A:=\bigoplus _{i\ges 0} A^{i}$ will be a
finitely generated graded commutative $k$-algebra with $A^{0}=k$; we
have in mind $H^{*}(G,k)$, for a finite group scheme $G$ over $k$.

As usual $\Proj A$ denotes the homogeneous prime ideals in $A$ that do
not contain the ideal $A^{\ges 1}$. Given a point $\fp$ in $\Proj A$,
we write $k(\fp)$ for the graded residue field at $\fp$; this is the
homogeneous field of fractions of the graded domain $A/\fp$. Observe
that $k(\fp)^{0}$ is a field extension of $k$ and $k(\fp)$ is of the
form $k(\fp)^{0}[t^{\pm 1}]$ for some indeterminate $t$ over
$k(\fp)^{0}$; see, for example,
\cite[Lemma~1.5.7]{Bruns/Herzog:1998a}.

\begin{lemma}
\label{lem:residue} The degree of $k(\fm)^{0}/k$ is finite for any
closed point $\fm$ in $\Proj A$.
\end{lemma}

\begin{proof} 
  One way to verify this is as follows: The Krull dimension of $A/\fm$
  is one so, by Noetherian normalisation, there exists a subalgebra
  $k[t]$ of $A/\fm$ where $t$ is an indeterminate over $k$ and the
  $A/\fm$ is finitely generated $k[t]$-module. Thus, inverting $t$,
  one gets that $(A/\fm)_{t}$ is a finitely generated module over the
  graded field $k[t^{\pm 1}]$, and hence isomorphic to $k(\fm)$. The
  finiteness of the extension $k(\fm)/k[t^{\pm 1}]$ implies that the
  extension $k(\fm)^{0}/k$ of fields is finite.
\end{proof}

The result below is familiar; confer \cite[Proposition
3.6.16]{Bruns/Herzog:1998a}.

\begin{lemma}
\label{lem:hull} Let $A$ be as above, let $\fm$ be a closed point in
$\Proj A$ and set $R:=A_{\fm}$. The $R$-submodule $I:=\bigcup_{i\ges
0}\Hom^{*}_{k}(R/\fm^{i},k)$ of $\Hom_{k}^{*}(R,k)$ is the injective
hull of $k(\fm)$, and for any $\fm$-torsion $R$-module $N$, there is a
natural isomorphism
\[ \Hom_{R}(N,I)\cong \Hom_{k}(N,k)\,.
\]
\end{lemma}

\begin{proof} Set $K=k(\fm)^{0}$ and recall that $k(\fm)=K[t^{\pm
1}]$, for some indeterminate $t$ over $K$. Thus, one has isomorphisms
of graded $k(\fm)$-modules
\begin{align} \Hom_{k}^{*}(k(\fm),k) & \cong
\Hom_{K}^{*}(k(\fm),\Hom_{k}(K,k)) \label{eq:field-duality}\\ & \cong
\Hom_{K}^{*}(k(\fm),K) \notag \\ & \cong k(\fm)\,. \notag
\end{align} The first isomorphism is adjunction, the second holds
because $\rank_{k}K$ is finite, by Lemma~\ref{lem:residue}, and the
last one is a direct verification.

The $R$-module $\Hom_{k}^{*}(R,k)$ is injective and hence so is its
$\fm$-torsion submodule
\[ \bigcup_{i\ges 0}\Hom_{R}^{*}(R/\fm^{i},\Hom_{k}^{*}(R,k))\,.
\] This is precisely the $R$-module $I$, by standard adjunction. Thus
$I$ must be a direct sum of shifts of injective hulls of $k(\fm)$. It
remains to verify that $I$ is in fact just the injective hull of
$k(\fm)$. To this end, note that for any $\fm$-torsion $R$-module $N$,
one has isomorphisms of graded $k(\fm)$-modules
\begin{align*} \Hom_{R}^{*}(N,I) &\cong
\Hom_{R}^{*}(N,\Hom_{k}^{*}(R,k)) \\ & \cong \Hom_{k}^{*}(N,k)\,.
\end{align*} This settles the last assertion in the desired result and
also yields the first isomorphism below of graded $k(\fm)$-modules.
\[ \Hom_{R}^{*}(k(\fm),I) \cong \Hom_{k}^{*}(k(\fm),k) \cong k(\fm)
\] The second one is by \eqref{eq:field-duality}. It follows that $I$
is the injective hull of $k(\fm)$.
\end{proof}

The next result, whose proof is rather similar to the one above, gives
yet another way to get to the injective hull at a closed point of
$\Proj$.

Recall that $I(\fp)$ denotes the injective hull of $A/\fp$ for any
$\fp$ in $\Spec A$.

\begin{lemma}
\label{lem:hull2} Let $A\to B$ be a homomorphism of graded commutative
algebras, let $\fm$ be a closed point in $\Proj B$, and set
$\fp:=\fm\cap A$. If the extension of residue fields $k(\fp)\subseteq
k(\fm)$ is finite, then the $B$-module
$\gam_{\mcV(\fm)}\Hom^{*}_{A}(B,I(\fp))$ is the injective hull of
$B/\fm$, and for any $\fm$-torsion $B$-module $N$, adjunction induces
an isomorphism
\[ \Hom_{B}(N,I(\fm)) \cong \Hom_{A}(N,I(\fp))\,.
\]
\end{lemma}

\begin{proof} 
  The $B$-module $I:=\gam_{\mcV(\fm)}\Hom^{*}_{A}(B,I(\fp))$ is
  injective, for it is the $\fm$-torsion submodule of the injective
  $B$-module $\Hom^{*}_{A}(B,I(\fp))$. As $\fm$ is a closed point, $I$
  is a direct sum of shifts of copies of $I(\fm)$. It remains to make
  the computation below:
\begin{align*} 
  \Hom^{*}_{B}(k(\fm),I) &\cong \Hom^{*}_{B}(k(\fm),
                           \Hom^{*}_{A}(B,I(\fp))) \\ 
                         &\cong \Hom^{*}_{A}(k(\fm),I(\fp)) \\
                         &\cong
                           \Hom^{*}_{k(\fp)}(k(\fm),\Hom^{*}_{A}(k(\fp),I(\fp))\\ 
                         &\cong
                           \Hom^{*}_{k(\fp)}(k(\fm),k(\fp)) \\ 
                         &\cong k(\fm)
\end{align*} 
These are all isomorphisms of $k(\fm)$-modules. The last
one is where the hypothesis that $k(\fm)/k(\fp)$ is finite is
required. This implies that $I\cong I(\fm)$. Given this, the last
isomorphism follows by standard adjunction.
\end{proof}

\begin{ack} 
  It is a pleasure to thank Tobias Barthel for detailed comments on an
  earlier version of this manuscript.
\end{ack}


\begin{thebibliography}{10}

\bibitem{Auslander:1978a} M.~Auslander, \emph{Functors and morphisms
determined by objects}, Representation theory of algebras
(Proc. Conf., Temple Univ., Philadelphia, Pa., 1976), 1--244. Lecture
Notes in Pure Appl. Math., 37, Dekker, New York, 1978.

\bibitem{Barthel/Heard/Valenzuela:2016a} T.~Barthel, D.~Heard, and
G.~Valenzuela, \emph{Local duality for structured ring spectra}, J.~Pure Appl. Algebra, to appear; \texttt{arXiv:1608.03135v1}

\bibitem{Benson:1984a} D.~Benson, \emph{Modular representation theory:
New trends and methods}, Lecture Notes in Mathematics \textbf{1081},
Springer-Verlag, Berlin/NewYork, 1984.

\bibitem{Benson:2001a} D.~Benson, \emph{Modules with injective
cohomology, and local duality for a finite group}, New York
J. Math. \textbf{7} (2001), 201--215.

\bibitem{Benson:2008a} D.~Benson, \emph{Idempotent $kG$-modules with
injective cohomology}, J.~Pure Appl. Algebra \textbf{212} (2008),
1744--1746.

\bibitem{Benson/Carlson/Rickard:1996a} D. J. Benson, J. F. Carlson,
and J. Rickard, \emph{Complexity and varieties for infinitely
generated modules. II}, Math. Proc. Cambridge
Philos. Soc. \textbf{120} (1996), 597--615.

\bibitem{Benson/Greenlees:2008a} D.~Benson and J.~P.~C.~Greenlees,
\emph{Localization and duality in topology and modular representation
theory}, J.~Pure Appl. Algebra \textbf{212} (2008), 1716--1743.

\bibitem{Benson/Greenlees:2014a} D.~Benson and J.~P.~C.~Greenlees,
\emph{Stratifying the derived category of cochains on BG for G a
compact Lie group}, J.~Pure Appl. Algebra \textbf{218} (2014),
642--650.

\bibitem{Benson/Iyengar/Krause:2008a} D.~J. Benson, S.~B. Iyengar, and
H.~Krause, \emph{{Local cohomology and support for triangulated
categories}}, Ann.\ Scient.\ \'Ec.\ Norm.\ Sup.\ (4) \textbf{41}
(2008), 575--621.

\bibitem{Benson/Iyengar/Krause:2011a} D.~J. Benson, S.~B. Iyengar, and
H.~Krause, \emph{Stratifying triangulated categories}, J.~Topology
\textbf{4} (2011), 641--666.

\bibitem{Benson/Iyengar/Krause:2012a} 
D.~J. Benson, S.~B. Iyengar, and H.~Krause,
\emph{Representations of finite groups: Local cohomology and support},
Oberwolfach Seminar \textbf{43}, Birkha\"user 2012.

\bibitem{Benson/Iyengar/Krause:2015a} D.~J. Benson, S.~B. Iyengar, and
H.~Krause, \emph{A local-global principle for small triangulated
categories}, Math. Proc. Camb. Phil. Soc.  \textbf{158} (2015),
451--476.

\bibitem{Benson/Iyengar/Krause/Pevtsova:2016a} D.~J. Benson,
S.~B. Iyengar, H.~Krause, and J.~Pevtsova, \emph{Colocalising
subcategories of modules over finite group schemes}, Ann. K-theory, 
\textbf{2} (2017), 387--408.

\bibitem{Benson/Iyengar/Krause/Pevtsova:2015b} D.~J. Benson,
S.~B. Iyengar, H.~Krause, and J.~Pevtsova, \emph{Stratification for
module categories of finite group schemes}, J.~Amer.~Math.~Soc., to appear;
\texttt{arXiv:1510.06773v2}

\bibitem{Benson/Krause:2002a} D.~J.~Benson and H.~Krause, \emph{Pure
injectives and the spectrum of the cohomology ring of a finite group},
J. Reine Angew. Math.  \textbf{542} (2002), 23--51.

\bibitem{Bondal/Kapranov:1990a} A. I. Bondal\ and\ M. M. Kapranov,
\emph{Representable functors, Serre functors, and mutations},
Izv. Akad. Nauk SSSR Ser. Mat. \textbf{ 53} (1989), no. 6, 1183--1205,
1337; translation in Math. USSR-Izv. \textbf{ 35} (1990), no.~3,
519--541.

\bibitem{Bruns/Herzog:1998a} W.~Bruns and J.~Herzog,
\emph{{Cohen--Macaulay rings}}, Cambridge Studies in Advanced
Mathematics, vol.~39, Cambridge University Press, 1998, Revised
edition.

\bibitem{Carlson:1996a} J.~F.~Carlson, \emph{{Modules and Group
Algebras}}, Lectures in Mathematics, ETH Z\"urich, Birkh\"auser, 1996.

\bibitem{Cartan/Eilenberg:1956a} H. Cartan\ and\ S. Eilenberg,
\emph{Homological algebra}, Princeton Univ. Press, Princeton, NJ,
1956.

\bibitem{Crawley-Boevey:1991a} W. Crawley-Boevey, \emph{Tame algebras
and generic modules}, Proc. London Math. Soc. (3) \textbf{63} (1991),
241--265.

\bibitem{Crawley-Boevey:1992a} W. Crawley-Boevey, \emph{Modules of
finite length over their endomorphism rings}, Representations of
algebras and related topics (Kyoto, 1990), 127--184, London
Math. Soc. Lecture Note Ser., 168 Cambridge Univ. Press, Cambridge,
1992.

\bibitem{Farnsteiner:2007a} R. Farnsteiner, \emph{Support spaces and
Auslander-Reiten components}, Vertex Operator Algebras and Their
Applications, Contemp. Math. \textbf{442} (2007), 61--87.

\bibitem{Friedlander/Suslin:1997a} E.~M. Friedlander and A.~Suslin,
\emph{{Cohomology of finite group schemes over a field}}, Invent.\
Math. \textbf{127} (1997), 209--270.

\bibitem{Happel:1988a} D.~Happel, \emph{{Triangulated categories in
the representation theory of finite dimensional algebras}}, London
Math.\ Soc.\ Lecture Note Series, vol.  119, Cambridge University
Press, 1988.

\bibitem{Happel:1991a} D. Happel, \emph{On Gorenstein algebras},
Representation theory of finite groups and finite-dimensional algebras
(Bielefeld, 1991), 389--404, Progr. Math., 95, Birkh\"auser, Basel,
1991.

\bibitem{Happel/Preiser/Ringel:1982a} D. Happel, U. Preiser, and
C. M. Ringel, \emph{Vinberg's characterization of Dynkin diagrams
using subadditive functions with application to $D{\rm Tr}$-periodic
modules}, Representation theory, II (Proc. Second Internat. Conf.,
Carleton Univ., Ottawa, Ont., 1979), 280--294, Lecture Notes in Math.,
832, Springer, Berlin, 1982

\bibitem{Hartshorne:1977a} R.~Hartshorne, \emph{{Algebraic geometry}},
Graduate Texts in Mathematics, vol.~52, Springer-Verlag, Ber\-lin/New
York, 1977.

\bibitem{Jantzen:2003a} J.~C. Jantzen, \emph{{Representations of
algebraic groups}}, American Math.\ Society, 2003, 2nd ed.

\bibitem{Krause:2000a} H. Krause, \emph{Auslander-Reiten theory via
Brown representability}, $K$-Theory \textbf{20} (2000), 331--344.

\bibitem{Krause:2003a} H.~Krause, \emph{A short proof of Auslander's
defect formula}, Linear Alg. Appl., \textbf{365} (2003), 267--270.

\bibitem{Krause:2005a} H.~Krause, \emph{The stable derived category of
a noetherian scheme}, Compositio Math.  \textbf{141} (2005),
1128--1162.

\bibitem{Krause/Le:2006a} H. Krause\ and\ J. Le, \emph{The
Auslander-Reiten formula for complexes of modules},
Adv. Math. \textbf{207} (2006), 133--148.

\bibitem{Krause/Reichenbach:2000a} H. Krause\ and\ U. Reichenbach,
\emph{Endofiniteness in stable homotopy theory},
Trans. Amer. Math. Soc. \textbf{353} (2000), 157--173.

\bibitem{Liu:2010a} S. Liu, \emph{Auslander-Reiten theory in a
Krull-Schmidt category}, S\~{a}o Paulo J. Math. Sci. \textbf{4}
(2010), 425--472.

\bibitem{Ravenel:1992a} D.~C.~Ravenel, \emph{Nilpotence and
periodicity in stable homotopy theory}, Annals of Mathematics Studies,
\textbf{128}, Princeton University Press, Princeton, NJ, 1992.

\bibitem{Reiten/VandenBergh:2002a} I. Reiten\ and\ M. Van den Bergh,
\emph{Noetherian hereditary abelian categories satisfying Serre
duality}, J. Amer. Math. Soc. \textbf{15} (2002), no.~2, 295--366.

\bibitem{Skowronski/Yamagata:2011a} A. Skowro\'nski\ and\ K. Yamagata,
\emph{Frobenius algebras. I}, EMS Textbooks in Mathematics, European
Mathematical Society (EMS), Z\"urich, 2011.

\bibitem{Waterhouse:1979a} W. C. Waterhouse, \emph{Introduction to
affine group schemes}, Graduate Texts in Mathematics, 66, Springer,
New York, 1979.

\end{thebibliography}
\end{document}